\documentclass[10pt]{amsart}
\usepackage{amsmath,amssymb,amsfonts,amsthm,amsopn}
\usepackage{latexsym,graphicx}
\usepackage{xcolor}
\usepackage{color, colortbl}
\usepackage{pb-diagram}
\usepackage[title]{appendix}
\usepackage{tikz-cd}
\usepackage{tikz}
\usepackage{mathtools}

\mathtoolsset{showonlyrefs}

\usepackage{multirow}

\newtheorem{theorem}{Theorem}[section]
\newtheorem{lemma}[theorem]{Lemma}
\newtheorem{corollary}[theorem]{Corollary}
\newtheorem{proposition}[theorem]{Proposition}

\newtheorem{remark}[theorem]{Remark}
\theoremstyle{definition}
\newtheorem{example}[theorem]{Example}

\newcommand{\beqa}{\begin{eqnarray*}}
\newcommand{\eeqa}{\end{eqnarray*}}

\DeclareMathOperator*{\Sp}{Sp}
\DeclareMathOperator*{\Mp}{Mp}
\DeclareMathOperator*{\Sym}{Sym}
\DeclareMathOperator*{\diag}{diag}
\DeclareMathOperator*{\GL}{GL}

\newcommand{\field}[1]{\mathbb{#1}}
\newcommand{\bR}{\field{R}}        
\newcommand{\bN}{\field{N}}        
\newcommand{\bC}{\field{C}}        
\newcommand{\cL}{\mathcal{L}}     %
\def\cF{\mathcal{F}}              
\def\cS{\mathcal{S}}
\def\cD{\mathcal{D}}

\def\cM{\mathcal{M}}

\def\cA{\mathcal{A}}
\def\cJ{\mathcal{J}}

\def\a{\aleph}

\def\rd{\bR^d}

\def\rdd{{\bR^{2d}}}

\def\intrdd{\int_{\rdd}}

\def\R{\right)}

\def\<{\left<}
\def\>{\right>}

\def\mv1{M_v^1}

\def\phas{(x,\xi )}
\def\mn{(m,n)}
\def\mn'{(m',n')}

\newcommand{\norm}[1]{\lVert#1\rVert}

\def\c{\hfill\break}

\def\a{\alpha}

\def\R{\mathbb{R}}
\def\Ren{\mathbb{R}^d}

\def\Sn2{S_{2}(L^{2}(\Ren))}
\def\S1{S_{1}(L^{2}(\Ren))}
\def\sig00{\sigma_{0,0}}

\begin{document}

\begin{abstract} 
Hardy's uncertainty principle is a classical result in harmonic analysis, stating that a function in $L^2(\rd)$ and its Fourier transform cannot both decay arbitrarily fast at infinity. 
In this paper, we extend this principle to the propagators of Schr\"odinger equations with quadratic Hamiltonians, known in the literature as metaplectic operators. These operators generalize the Fourier transform and have captured significant attention in recent years due to their wide-ranging applications in time-frequency analysis, quantum harmonic analysis, signal processing, and various other fields. However, the involved structure of these operators requires careful analysis, and 
{most results obtained so far concern special propagators that can basically be reduced to rescaled Fourier transforms}. The main contributions of this work are threefold: (1) we extend Hardy's uncertainty principle, covering all propagators of Schr\"odinger equations with quadratic Hamiltonians, (2) we provide concrete examples, such as fractional Fourier transforms, which arise when considering anisotropic harmonic oscillators, (3) we suggest Gaussian decay conditions in certain directions only, which are related to the geometry of the corresponding Hamiltonian flow.
\end{abstract}

\title[Hardy's Uncertainty principle]{Hardy's uncertainty principle for Schr\"odinger equations with quadratic Hamiltonians}

\author{Elena Cordero}
\address{Universit\`a di Torino, Dipartimento di Matematica, via Carlo Alberto 10, 10123 Torino, Italy}
\email{elena.cordero@unito.it}
\author{Gianluca Giacchi}
\address{Euler Institute, Università della Svizzera Italiana, Dipartimento di Informatica, Via La Santa 1, 6900 Lugano, Switzerland}
\email{gianluca.giacchi@usi.ch}
\author{Eugenia Malinnikova}
\address{Stanford University, 450 Jane Stanford Way, Stanford, California, 94305, US $\&$ \newline
Norwegian University of Science and Technology, NO-7491, Trondheim, Norway}
\email{eugeniam@stanford.edu}

\thanks{}
\maketitle

%
%
%

\section{Introduction}
Uncertainty principles are classical results in harmonic analysis stating that, whenever a meaningful definition of {\em localization} is given, a function $f\in L^2(\rd)$ and its Fourier transform $\hat f$ cannot be both well-localized in their respective domains. The notion of localization we consider in the present work is {\em Gaussian decay}, corresponding to {\em Hardy's uncertainty principle}, formulated in \cite{Hardy} by Hardy for functions in $L^2(\bR)$, and later generalized by Sitaram, Sundari and Thangavelu to $L^2(\rd)$ in the following synthesized form, \cite{SST}.

\begin{theorem}\label{thmIntro1}
	Let $f\in L^2(\rd)$, and $a,b>0$. Assume that:
	\begin{align*}
		&|f(x)|\lesssim e^{-\pi a |x|^2}, \qquad x\in\rd,\\
		&|\hat f(\xi)|\lesssim e^{-\pi b|\xi|^2}, \qquad \xi\in\rd.
	\end{align*}
	If $ab>1$, then $f\equiv0$. If $ab=1$ then $f(x)=Ce^{-\pi a|x|^2}$.
\end{theorem}

 This result is \emph{isotropic}, as the decay of $f$ and $\hat f$ is assumed to be the same along every direction. 
In the last decades, Theorem \ref{thmIntro1} underwent several generalizations, see for example  \cite{BDsurvey,BDJ}. Various uncertainty principles, including the one of Hardy, have been extended  to the setting of metaplectic operators with free symplectic projections in \cite{CGP2024,Liu,Z1}.

Metaplectic operators were introduced and studied by several authors in the last century \cite{Segal, Shale, VHover, WeylM}, and they captured the attention of researchers in the last decades due to their applications to signal analysis, PDEs, time-frequency  and quantum harmonic analysis \cite{CGP2024,GZ2001,MO2002,ZH}. Despite their very algebraic definition, the simplest way to view them analytically is as compositions (up to phase factors) of three elementary operators: the Fourier transform itself,
\begin{equation*}
	\hat f(\xi)=\int_{\rd}f(x)e^{-2\pi i\xi\cdot x}dx, \qquad f\in L^1(\rd),
\end{equation*} 
the product operators:
\begin{equation}\label{prodintro}
	\mathfrak{p}_Qf(t)=e^{i\pi Qt\cdot t}f(t), \qquad f\in L^2(\rd),
\end{equation}
($Q\in\bR^{d\times d}$ symmetric) and the rescalings:
\begin{equation}\label{rescintro}
	\mathfrak{T}_Ef(t)=|\det(E)|^{1/2}f(Et), \qquad f\in L^2(\rd),
\end{equation}
($E\in\bR^{d\times d}$ invertible). The straightforward expression of these three generators should not mislead the reader, as metaplectic operators are not limited to those non-trivial examples, covering operators such as fractional Fourier transforms, and propagators of Schr\"odinger equations with quadratic Hamiltonians. On the top of that, metaplectic operators also play a fundamental role in the representation of Schr\"odinger propagators of equations with perturbed Hamiltonians, we refer the interested reader to \cite{CNR}.

Any metaplectic operator $\hat S$ is naturally related to a symplectic matrix $S$, its \emph{projection}, which is usually studied in terms of its $d\times d$ block decomposition:
\begin{equation}\label{blockSintro}
S=\begin{pmatrix}A & B\\ C & D\end{pmatrix}\in\bR^{2d\times2d}, \qquad A,B,C,D\in\bR^{d\times d}.
\end{equation}
$S$ is called \emph{free} if the block $B$ is invertible. In view of this relation, metaplectic operators are relatively easy to handle and decompose, and the question arises whether it is possible to infer their properties from the structure of the corresponding projection.

Hardy's uncertainty principle for metaplectic operators with free projections \cite[Theorems 23 and 28]{CGP2024} highlights that the interplay between the directions where the function $f$ decays exponentially and the invertible block $B$ is encoded in the matrix $MB^TNB$. Precisely:
\begin{theorem}\label{thmIntro3}
	Let $\hat S$ be a metaplectic operator with free symplectic projection $S$, and block decomposition \eqref{blockSintro}. Let  $M,N\in\Sym(d,\bR)$ be positive-definite matrices and $f\in L^2(\rd)$ be such that:
	\begin{align*}
		&|f(x)|\lesssim e^{-\pi Mx\cdot x}, \qquad x\in\rd, \\
		&|\hat Sf(\xi)|\lesssim e^{-\pi N\xi\cdot\xi}, \qquad \xi\in\rd.
	\end{align*}
	If $\lambda>1$ for some eigenvalue $\lambda$ of $MB^TNB$, then $f\equiv0$.
\end{theorem}
The metaplectic operators with free symplectic projections can be viewed as rescaled Fourier transforms (see Remark \ref{freeOP} below). The known generalizations of Hardy's uncertainty principles to metaplectic operators employ this connection and are restricted to operators with free symplectic projection.

The first aim of this work is to generalize Hardy's uncertainty principle to every metaplectic operator, regardless of the invertibility of the block $B$. Straightforward examples, as the following, show that in these cases Hardy's uncertainty principle features a directional selectivity depending on the block $B$ in \eqref{blockSintro}. 
\begin{example}\label{exIntro1}
\begin{figure}\label{fig:1}
	\begin{minipage}[h]{\textwidth}
		\vspace{1cm}
		\includegraphics[width=\textwidth]{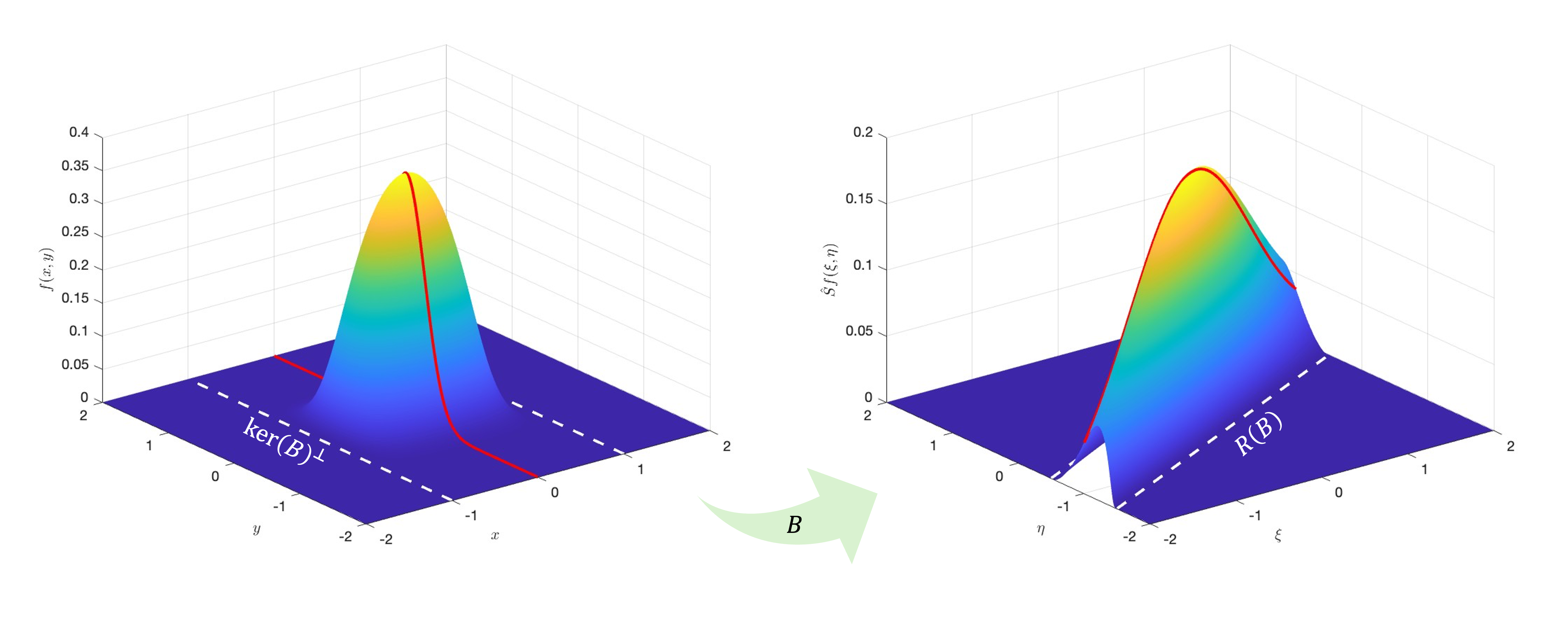}
	\end{minipage}\hfill
	\caption{The directional selectivity of Hardy's uncertainty principle for metaplectic operators, involving $\ker(B)^\perp$ and $R(B)$, illustrated in Example \ref{exIntro1}.} 
\end{figure} 
	Consider the metaplectic operator 
	\[
		\hat Sf(\xi,\eta)=\cF_2f(-\xi+2\eta,-\xi+\eta), 
	\]
	where $\cF_2F(x,\xi)=\int_{\bR^2}F(x,y)e^{-2\pi iy \xi}dy$ is the partial Fourier transform with respect to the second variable. 
	The associated symplectic matrix
	\[
		S=\left(\begin{array}{cc|cc}
				1 & 0 & 0 & -2\\
				1 & 0 & 0 & -1\\
			\hline
				0 & 1 & -1 & 0\\
				0 & -1 & 2 & 0
		\end{array}\right),
	\]
	has 
	\[
		\ker(B)^\perp=\mbox{span}\left\{\begin{pmatrix} 0 \\ 1\end{pmatrix}\right\}, \quad \text{and}\quad R(B)=\mbox{span}\left\{\begin{pmatrix} 2 \\ 1\end{pmatrix}\right\}.
	\]
	Consider the function 
	\[
		f(x,y)=\varphi(x)e^{-2\pi y^2},
	\]
	represented in Figure \ref{fig:1} (a), where $\varphi\in\mathcal{C}^\infty(\bR)$ is supported on $[-1,1]$. Then,
	\[
		\hat Sf(\xi,\eta)=\frac{1}{2}\varphi(-\xi+2\eta)e^{-\pi (-\xi+\eta)^2/2},
	\]
	that we represented in Figure \ref{fig:1} (b). Observe that $f$ is supported on the cylinder $([-1,1]\times\{0\})+\ker(B)^\perp$, whereas $\hat Sf$ is supported on $(\{0\}\times[-1,1])+R(B)$. 
\end{example}

Theorem \ref{thmIntro3} cannot be extended to metaplectic operators with symplectic projection having $B=0_{d}$. In fact, a stronger conclusion can be inferred:
\begin{proposition}\label{propintro1}
	Let $\hat S$ be a metaplectic operator with $B=0_{d}$. Then, there exists $f\in L^2(\rd)\setminus\{0\}$ such that $f$ and $\hat Sf$ have compact support.
\end{proposition} 
In the context of Hardy's uncertainty principle, the question arises whether there are other non-elementary cases where the principle fails.
In this paper, we answer this question by showing that, whenever $B\neq0_{d}$, Hardy's uncertainty principle holds. In particular, if also $B\notin \GL(d,\bR)$, the directions in $\ker(B)^\perp$ and $R(B)$ play a special role. Let us state our main result, cf. Theorem \ref{Hardyglobale} in Section \ref{sec:HUP} and the discussion at the beginning of Section \ref{sec:sharpness}.
\begin{theorem}\label{thmIntro4}
	Let $\hat S\in\Mp(d,\bR)$ be a metaplectic operator, and assume that $B\neq0_{d}$. Let $M$ and $N$ be positive-semidefinite matrices with
	\begin{equation}\label{kerMRN}
		\ker(M)=\ker(B) \qquad and \qquad R(N)=R(B).
	\end{equation}
	Let $f\in L^2(\rd)$ satisfy the decay estimates 
	\begin{align}
	\label{G1-2 intro}
	& |f(x)|\lesssim e^{-\pi Mx\cdot x}, \qquad a.e. \ x\in\rd,\\
	\label{G2-2 intro}
	& |\hat Sf(\xi)|\lesssim e^{-\pi N\xi\cdot\xi}, \qquad a.e. \ \xi\in\rd.
\end{align}
If $\lambda>1$ for some eigenvalue $\lambda$ of $MB^TNB$, then $f=0$ almost everywhere.

\end{theorem}

	We point out that under the above conditions \eqref{kerMRN}, the matrix $MB^TNB$ is an isomorphism of $\ker(B)^\perp$. We also remind that $MB^TNB$ is the product of two positive semi-definite matrices and all its eigenvalues are real and non-negative.
  To conclude our presentation of Hardy's uncertainty principle for metaplectic operator, we prove that the directional decay in Theorem \ref{thmIntro4} is sharp, as detailed in Theorem \ref{sharp} below. The next example demonstrates the connection of Theorem \ref{thmIntro4} to the classical Hardy's uncertainty principle. 
  
  \begin{example} Let $L=\ker(B)$ and $R=R(B)$. We denote by $P$ and $Q$ the orthogonal projections onto $L^\perp$ and $R$ respectively. Then, applying the above theorem for $M=aP$ and $N=bQ$, we see that the decay assumptions \eqref{G1-2 intro} and \eqref{G2-2 intro} imply $f\equiv 0$ when $ab>\sigma(B)^{-2}$, where $\sigma(B)$ is the largest singular value of $B$.
  \end{example}
	
%
Since the propagators of Schr\"odinger equations with quadratic Hamiltonians are metaplectic operators, the theory developed so far finds important applications in the so-called dynamical version of Hardy's uncertainty principle, which involves the Schr\"{o}dinger evolution.
In the context of quantum mechanics, the Schr\"odinger equation is fundamental for describing the time evolution of a quantum state. 
The analysis of quadratic Hamiltonians using symplectic and metaplectic groups provides deep insights into the solutions of the Schr{\"o}dinger equation. 
The Hamiltonian  can be represented using a real-valued symmetric matrix \(\mathcal{M} \):
\begin{equation}\label{Hamilt}
	H(z) = \frac{1}{2} \langle \mathcal{M}z, z \rangle,
\end{equation}
where \( z = (x, \xi) \in\rdd\).
Symplectic mechanics provides the framework for analyzing systems with quadratic Hamiltonians. 
For a quadratic Hamiltonian, the time evolution of the quantum state \( u(x, t) \) is governed by:
\begin{equation}\label{eqScrod}
 i \hbar \frac{\partial u}{\partial t} (x, t) = H_D u(x, t), 
\end{equation}
where $H_D$ is the Weyl quantization of $H$:
\[
H_Df(x)=\intrdd e^{2\pi i x\cdot\xi} H\left(\frac{x+y}{2},\xi\right) f(y)\,dyd\xi.
\]
(see Section \ref{sec:DH} for details).
If we consider the equation \eqref{eqScrod} with the initial condition \( u(x, 0) = u_0(x) \), the solution can be expressed in terms of a metaplectic operator \( \hat S_H^t:=e^{itH} \) acting on the initial state \( u_0(x) \):

\[ u(x, t) =\hat S_H^t u_0(x).\]
The so-called dynamical version of the Hardy's uncertainty principle establishes conditions under which the solution must vanish, highlighting the interplay between symplectic geometry and quantum mechanics. For a comprehensive interpretation of this dynamical version we refer to the recent survey  \cite{FM2021}, and  to the pioneering works with $H_D=\Delta+V$ in \cite{Chanillo,EKPV2006}, see also  \cite{CEKPV2010}.
More recently, Cassano and Fanelli  \cite{CF2015,CF2017} studied the special cases of the harmonic oscillator  and of systems with a magnetic potential {as well as bounded perturbations of such systems}. They have proved uniqueness results for the Hamiltonian $H_D=\Delta_A+V(x,t)$, where the (electro)magnetic Laplacian is  $\Delta_A=(\nabla-iA(x))^2$, with magnetic potential given by some coordinate transformation $A:\rd\to\rd$. 
 Knutsen in \cite{Knutsen} proved new results in this framework  using the Hardy's uncertainty principle  for the Wigner distribution.
 
 Our main result in this direction is Theorem \ref{main} below, which  generalizes Theorem 3.1 in \cite{Knutsen} in two ways:\\
 (i) It gives sharp sufficient conditions for the uncertainty principle.\\
 (ii) It works for every symplectic matrix with block $B\not=0_{d}$.
 
The basic idea of Theorem \ref{main} is Hardy's uncertainty principle for metaplectic operators in Theorem \ref{Hardyglobale} which makes the proof of Theorem \ref{main} decidedly simple.
To give a flavor of this result, we state here the simplified case of a free symplectic matrix ($\det B\not=0$).
\begin{theorem}\label{1.7}
Let \( u(x, t)  \) be the solution to the Schrödinger equation \eqref{eqScrod} with a quadratic Hamiltonian \eqref{Hamilt}, $\hbar=1/(2\pi)$, and initial datum $u_0\in \mathcal{S}(\mathbb{R}^d)$. Consider $M,N\in\Sym(d,\bR)$  positive-definite matrices. 
Suppose at times \( t = 0 \) and \( t = t_1 \) the solution satisfies:
\begin{equation}\label{E7}
	 |u(x, 0)| \lesssim  e^{-\pi Mx\cdot x}, \quad x\in \rd, \quad \text{and} \quad |u(x, t_1)|  \lesssim  e^{-\pi Nx\cdot x}, \quad x\in \rd.
\end{equation}
Assume that the symplectic projection \( S^{t_1}_H \) of the operator which satisfies $u(x,t_1)=\hat S^{t_1}_H u_0$ is free and has block decomposition \eqref{blockSintro} with block $B=B(t_1)$. Let $\lambda_1,\ldots,\lambda_d$ be (positive) eigenvalues of $MB^T(t_1)NB(t_1)$. If there exists a $ j\in\{1,\dots,d\}$ such that $\lambda_j>1$, then $u_0\equiv0$.
\end{theorem}

{\bf Outline.} The rest of this work is organized as follows: Section \ref{sec:2} introduces the notation and main theory used. Section \ref{sec:integRep}  presents the key tool for proving Hardy's uncertainty principle, detailed in Section \ref{sec:HUP}, accompanied 
with a discussion about directional selectivity and explicit non-trivial examples. Finally, Section \ref{sec:DH} covers the dynamic versions of Hardy's uncertainty principle. Some algebraic results are collected in the Appendices for the convenience of the reader. 

\section{Notation and preliminaries}\label{sec:2}
\subsection{ Linear algebra notation and change of variables}
The standard scalar product in $\rd$ is denoted by $xy=x\cdot y=x^Ty$, $x,y\in\rd$. The Euclidean norm on $\rd$ is denoted by $|\cdot|$, i.e., $|x|=\sqrt{x\cdot x}$. 

If $E\in\bR^{d\times d}$, $\mbox{ker}(E)$ denotes its kernel and $R(E)$ denotes its range. Recall that $R(E^T)=\ker(E)^\perp$. We also denote by $E^{-1}(X)$ the pre-image of a space $X\subseteq\rd$ under $E$. Moreover, if $E\in\bR^{m\times n}$, $E^+$ denotes its Moore-Penrose inverse.

The group of $d\times d$ invertible matrices is denoted by $\GL(d,\bR)$. By $\Sym(d,\bR)=\{P\in\bR^{d\times d}:P^T=P\}$ we denote the space of $d\times d$ symmetric matrices. Moreover, we denote by $\Sym_{++}(d,\bR)$ and $\Sym_+(d,\bR)$ the set of $d\times d$ positive-definite and positive-semidefinite matrices, respectively.


	If $\cL\subseteq\rd$ is a linear subspace of dimension $r$, and $E:\rd\to\rd$ with $E(\cL)$ having dimension $r$, then
	\[
		\int_{\cL}f(Ex)dx=\frac{1}{q_\cL(E)}\int_{E(\cL)}f(y)dy, 
	\]
	whenever $f:E(\cL)\to\bC$ is measurable and the integral converges. The constant 
	$q_\cL(E)$ is the volume of the simplex generated by $Ev_1,\ldots,Ev_r$, {with  $v_1,\ldots,v_r$ being} any orthonormal basis of $\cL$. If $\cL=\rd$ and $A\in\GL(d,\bR)$, then $q_\cL(E)=|\det(E)|$. 
	
	If $V:\bR^r\to \cL$ is a parametrization of $\cL$, then 
	\[
		\int_{\cL}f(x)dx=\sqrt{\det(V^TV)}\int_{\bR^r}f(Vu)du,
	\]
	\cite[Theorem 11.25]{Folland}. In particular, if $V$ maps the canonical basis of $\bR^r$ to an orthonormal basis of $\cL$, then $V^TV=I_{ r}$ and, consequently,
	\[
		\int_{\cL}f(x)dx=\int_{\bR^r}f(Vu)du.
	\]

\subsection{The symplectic group}
We denote by $\mbox{Sp}(d,\bR)$ the group of $2d\times 2d$ symplectic matrices. Specifically, a matrix $S\in\bR^{2d\times 2d}$ is symplectic if it has block decomposition:
\begin{equation}\label{blockS}
	S=\begin{pmatrix}
		A & B\\
		C & D
	\end{pmatrix}, 
\end{equation}
with the blocks $A,B,C,D\in \bR^{d\times d}$ satisfying:
\begin{equation}\label{sympRel}
	\begin{cases}
		A^TC=C^TA,\\
		B^TD=D^TB,\\
		A^TD-C^TB=I_{d}.
	\end{cases}
\end{equation}
We call $S$ \emph{free} if $B\in\GL(d,\bR)$. There are two other equivalent definitions of the symplectic group. Firstly, $S\in\Sp(d,\bR)$ if $S^TJS=J$, where
\begin{equation}\label{defJ}
	J=\begin{pmatrix}
		0_{d} & I_{d}\\
		-I_{d} & 0_{d}
	\end{pmatrix},
\end{equation} 
and, secondly, $S\in\Sp(d,\bR)$ is in the form \eqref{blockS} and $S^{-1}$ has block decomposition:
\begin{equation}\label{blockS-1}
	S^{-1}=\begin{pmatrix}
		D^T & -B^T\\
		-C^T & A^T
	\end{pmatrix}.
\end{equation}
\begin{remark}
	The symplectic relations \eqref{sympRel} can be rephrased by plugging the blocks \eqref{blockS-1} into \eqref{sympRel}:
	\begin{equation}\label{sympRel-1}
		\begin{cases}
			CD^T=DC^T,\\
			AB^T=BA^T,\\
			AD^T-BC^T=I_{d}.
		\end{cases}
	\end{equation}
	Observe that if $B$ is invertible, the second relation is equivalent to
	\[
		B^{-1}A=(B^{-1}A)^T.
	\]
\end{remark}

\begin{remark}
$\Sp(d,\bR)$ is generated by the matrices $J$, and
\begin{align}
	\label{defVP}
	&V_Q=\begin{pmatrix}
		I_{d} & 0_{d}\\
		Q & I_{d}
	\end{pmatrix}, \qquad Q\in\Sym(d,\bR),\\
	\label{defDL}
	&\cD_E=\begin{pmatrix}
		E^{-1} & 0_{d}\\
		0_{d} & E^T
	\end{pmatrix}, \qquad E\in\GL(d,\bR).
\end{align}
\end{remark}
\begin{example}
	(a) For $P\in\Sym(d,\bR)$, the upper block triangular matrix:
	\begin{equation}\label{defUQ}
		U_P=V_P^T=\begin{pmatrix}
			I_{d} & P\\
			0_{d} & I_{d}
		\end{pmatrix}
	\end{equation}
	is symplectic.\\
	(b) More generally, if $S\in \Sp(d,\bR)$, then also $S^T\in\Sp(d,\bR)$ and $S^{-1}\in\Sp(d,\bR)$.
\end{example}
\subsection{Metaplectic operators} For a classical introduction to metaplectic operators, we refer the reader to \cite{Elena-book,book}. In view of the techniques used in this work, we prefer to follow \cite{MO2002}, and define metaplectic operators in terms of the cross-Wigner distribution \cite{Elena-book,book}
\begin{equation}\label{WD}
	W(f,g)(x,\xi)=\int_{\rd}f\left(x+\frac{t}{2}\right)\overline{g\left(x-\frac{t}{2}\right)}e^{-2\pi i\xi\cdot t}dt \qquad \phas\in\rdd,
	\end{equation}
	$f,g\in L^2(\rd)$. Let $S\in \Sp(d,\bR)$. There exists a unitary operator $\hat S$ on $L^2(\rd)$ so that
	\begin{equation}\label{WD2}
		W(\hat Sf,\hat Sg)(x,\xi)=W(f,g)(S^{-1}\phas)
	\end{equation}
	holds for every $f,g\in L^2(\rd)$ and $\phas\in\rdd$. Any such operator is called {\em metaplectic operator}; 
$\hat S$ in \eqref{WD2} is uniquely determined up to a phase factor, meaning that if \eqref{WD2} is also satisfied by another operator $\hat S'$, then $\hat S'=c\hat S$ for some $c\in\bC$, $|c|=1$. The group $\{\hat S:S\in\Sp(d,\bR)\}$ admits a subgroup, denoted by $\Mp(d,\bR)$, consisting of exactly two operators for each symplectic matrix $S$, namely $\pm\hat S$. The projection $\pi^{Mp}:\hat S\in\Mp(d,\bR)\to S\in \Sp(d,\bR)$ is a group homomorphism with kernel $\ker(\pi^{Mp})=\{\pm\mbox{id}_{L^2}\}$. 

\begin{remark}\label{remMO}
	Observe that our notation is slightly different than the one adopted in \cite{MO2002}. Specifically, for $\cA\in \Sp(d,\bR)$ ter Morsche and Oonincx consider $\hat \cA$ so that 
	\begin{equation}\label{WD3}
		W(\hat \cA f,\hat \cA g)(x,\xi)=W(f,g)(\cA\phas),\qquad f,g\in L^2(\rd), \quad \phas\in\rdd,
	\end{equation}
	see \cite[Theorem 1]{MO2002}. By comparing \eqref{WD2} and \eqref{WD3}, we see that this is equivalent to using the projection $\widetilde{\pi^{Mp}}(\hat \cA)=\cA^{-1}$, instead of  $\pi^{Mp}(\hat\cA)=\cA$. Hence, the results in \cite{MO2002} are still valid in our framework, up to choosing
	\begin{equation}\label{AS}\cA=
		\begin{pmatrix}
			\cA_{11} & \cA_{12}\\
			\cA_{21} & \cA_{22}
		\end{pmatrix}=
		\begin{pmatrix}
			D^T & -B^T\\
			-C^T & A^T
		\end{pmatrix}
	\end{equation}
	in the statements and formulae therein.
\end{remark} 

\begin{example}\label{exMetap}
	(a) The Fourier transform $\cF$, defined for every $f\in \cS(\rd)$ by
	\begin{equation}
		\cF f(\xi)=\hat f(\xi)=\int_{\rd}f(x)e^{-2\pi i\xi\cdot x}dx, \qquad \xi\in\rd,
	\end{equation}
	is a metaplectic operator. Its projection is $\pi^{Mp}(\cF)=J$, defined in \eqref{defJ}.\\
	(b) If $Q\in\Sym(d,\bR)$,  \begin{equation}\label{defPhi}
			\Phi_Q(t)=e^{i\pi Qt\cdot t}
		\end{equation}
  is the corresponding chirp, and the operator
	\begin{equation}
		\mathfrak{p}_Qf(t)=\Phi_Q(t)f(t), \qquad f\in L^2(\rd),
	\end{equation}
	is metaplectic, with projection $\pi^{Mp}(\mathfrak{p}_Q)=V_Q$, defined in \eqref{defVP}. \\
	(c) If $E\in\GL(d,\bR)$, the rescaling operator:
	\begin{equation}
		\mathfrak{T}_Ef(t)=|\det(E)|^{1/2}f(Et), \qquad f\in L^2(\rd),
	\end{equation}
	is metaplectic, with projection $\pi^{Mp}(\mathfrak{T}_E)=\cD_E$, defined in \eqref{defDL}.\\
	(d) If $P\in\Sym(d,\bR)$, the multiplier operator:
	\begin{equation}
		\mathfrak{m}_Pf=\cF^{-1}(\Phi_{-P} \hat f), \qquad f\in L^2(\rd)
	\end{equation}
	is metaplectic, with $\pi^{Mp}(\mathfrak{m}_P)=U_P$, defined in \eqref{defUQ}.
\end{example} 
Other examples, such as fractional Fourier transforms, are displayed in the following sections.

\section{A representation formula for $\hat Sf$}\label{sec:integRep}

\subsection{Integral representation of ter Morsche and Oonincx} 
In this work, we will use the following integral representation proved by ter Morsche and Oonincx in \cite{MO2002}, that we reformulate with our notation. 
Let us consider the constant
\begin{equation}\label{definitionmuS}
		\mu_S=\sqrt{ \frac{1}{q_{R(B)^\perp}(A^T)\sigma(B)} },
\end{equation}
	where we recall that $q_{R(B)^\perp}(A^T)$ is the volume of the $(d-r)$-simplex spanned by the vectors $A^Tv_1,\ldots,A^Tv_{d-r}$, where $v_1,\ldots,v_{d-r}$ is any orthonormal basis of $R(B)^\perp$, whereas $\sigma(B)$ denotes the product of the non-zero singular values of $B$. \\
	
	For the remainder of this section, $\hat S\in\Mp(d,\bR)$ is a metaplectic operator with projection $S=\pi^{Mp}(\hat S)$ having block decomposition \eqref{blockS}. We will always assume that $1\leq rank(B)\leq d$.

\begin{proposition}\label{propCorretta}
	Let $S$ and $\hat S$ be as above. Then, for every $f\in L^2(\rd)$  we have:
	\begin{equation}\label{intrepformula}
		\hat Sf(\xi)=\textcolor{black}{\mu_S} e^{i\pi DC^T \xi\cdot\xi}\int_{\ker(B)^\perp}f(t+D^T\xi)e^{i\pi B^+At\cdot t}e^{2\pi iC^T\xi\cdot t}dt,
	\end{equation}
 where we recall that $B^+$ denotes the Moore-Penrose inverse of $B$. The formula above is understood as the equality of two $L^2(\R^d)$ functions.
\end{proposition}

\begin{proof}
	It is a restatement, using our notation, of the formula in Remark $(2)$ of \cite[Section 5]{MO2002}. 
	In fact, in Remark $(2)$ of \cite[Section 5]{MO2002} it is proved that
	\begin{equation}\label{intrepformula0}
		\hat Sf(\xi)=\sqrt{\frac{1}{q_{\ker(\cA_{12})}(\cA_{22})\sigma(\cA_{12})}} e^{-i\pi  \cA_{11}^T \cA_{21}\xi\cdot\xi}\int_{R( \cA_{12})}f(t+ \cA_{11}\xi)e^{-i\pi  \cA_{22} \cA_{12}^+t\cdot t}e^{-2\pi i \cA_{21}\xi\cdot t}dt,
	\end{equation}
	where, according to the notation of \cite{MO2002},
	\begin{equation}\label{blockS0}
		\cA=\begin{pmatrix}
			\c A_{11} &  \cA_{12}\\
			\cA_{21} & \cA_{22}
		\end{pmatrix}
	\end{equation}
	is the symplectic matrix such that
	\[
	W(\hat Sf,\hat Sg)(x,\xi)=W(f,g)(\cA(x,\xi)), \qquad f,g\in L^2(\rd),\,\, \phas\in\rdd.
	\]
	By Remark \ref{remMO}, we have $\cA=S^{-1}$, whose blocks are related to the blocks of $S$ by \eqref{blockS-1}. Thus, by \eqref{AS},
	\begin{align}
		\hat Sf(\xi)&=\sqrt{\frac{1}{q_{\ker(B^T)}(A^T)\sigma(B^T)}} e^{i\pi  D C^T\xi\cdot\xi}\int_{R( B^T)}f(t+D^T\xi)e^{-i\pi  A^T (B^T)^+t\cdot t}e^{2\pi i C^T\xi\cdot t}dt\\
		&=\sqrt{\frac{1}{q_{R(B)^\perp}(A^T)\sigma(B)}} e^{i\pi  D C^T\xi\cdot\xi}\int_{\ker(B)^\perp}f(t+D^T\xi)e^{-i\pi  B^+A\cdot t}e^{2\pi i C^T\xi\cdot t}dt,
	\end{align}
%
%
where we used that $\ker(B^T)=R(B)^\perp$ and $\sigma(B^T)=\sigma(B)$. This completes the proof.
	
\end{proof}

\subsection{Metaplectic operators as partial Fourier transforms}
Loosely speaking, the main tool to prove Hardy's uncertainty principle for metaplectic operators consists of an integral representation formula obtained by rewriting \eqref{intrepformula} as the Fourier transform of the restriction of $f$ to an affine subspace parallel to $\ker(B)^\perp$. To prove it, we need to decompose $\xi=\xi_1+\xi_2$, where $\xi_1\in R(B)$ and $\xi_2\in A(\ker(B))$. For expository reasons, we postpone the proof that such decomposition exists to Lemma \ref{lemma46} below.

\begin{corollary}\label{cor44}
	Let $f\in L^2(\rd)$  and $\xi=\xi_1+\xi_2$ be as above. Then, the following integral representation holds:
	\begin{align}
		\label{f2}
		\hat Sf(\xi)&{=}\mu_S e^{i\pi (DB^+ \xi_1\cdot\xi_1+DC^T\xi_2\cdot\xi_2)}\int_{\ker(B)^\perp}f(t+D^T\xi_2)e^{i\pi B^+At\cdot t}e^{-2\pi i (B^+\xi_1-C^T\xi_2)\cdot t}dt.
	\end{align}
	The formula above is understood as the equality of two $L^2(\R^d)$ functions.
\end{corollary}

\begin{remark}\label{freeOP}
	If $B\in\GL(d,\bR)$, then $\mu_S=|\det(B)|^{-1/2}$ and $B^+=B^{-1}$, so that formula \eqref{f2} reads as the well-known integral representation of metaplectic operators with free symplectic projections:
	\[
		\hat Sf(\xi)=|\det(B)|^{-1/2}e^{i\pi DB^{-1}\xi\cdot\xi}\int_{\rd}f(t)e^{i\pi B^{-1}At\cdot t}e^{-2\pi iB^{-1}\xi\cdot t}dt, \qquad \xi\in\rd.
	\]
\end{remark}

\begin{proof}[Proof of Corollary \ref{cor44}]	
	Starting from formula \eqref{intrepformula} in Proposition \ref{propCorretta}, and observing that the change of variables $u+D^T\xi_1=t$ ($\xi_1\in R(B)$) leaves $\ker(B)^\perp$ invariant by Corollary \ref{corB2} $(i)$, we have:
		\begin{align*}
			\hat Sf(\xi)&\underset{\eqref{intrepformula}}{=}\mu_S e^{i\pi DC^T \xi\cdot\xi}\int_{\ker(B)^\perp}f(u+D^T\xi_1+D^T\xi_2)e^{i\pi B^+Au\cdot u}e^{2\pi iC^T\xi\cdot u}du,\\
			&\underset{\text{Cor. \ref{corB2}$(i)$}}{=}\mu_S\underbrace{e^{i\pi DC^T \xi\cdot\xi}}_{(1)}\int_{\ker(B)^\perp}f(t+D^T\xi_2)\underbrace{e^{i\pi B^+A(t-D^T\xi_1) \cdot (t-D^T\xi_1)}}_{(2)}\underbrace{e^{2\pi iC^T\xi\cdot (t-D^T\xi_1)}}_{(3)}dt.
		\end{align*}
    Let us focus on the exponents of the chirps (1)--(3). Trivially, since $DC^T$ is symmetric, (1) becomes
    \[
DC^T\xi\cdot\xi=DC^T\xi_1\cdot\xi_1+DC^T\xi_2\cdot\xi_2+2DC^T\xi_1\cdot\xi_2.
    \]
    As far as (2) is concerned:
    \[
        B^+A(t-D^T\xi_1) \cdot (t-D^T\xi_1)=B^+At\cdot t-DB^+At\cdot\xi_1-B^+AD^T\xi_1\cdot t+DB^+AD^T\xi_1\cdot\xi_1.
    \]
    Finally, (3) is:
    \begin{align*}
   2C^T\xi\cdot (t-D^T\xi_1)&=2C^T(\xi_1+\xi_2)\cdot (t-D^T\xi_1)\\
    &=2(C^T\xi_1\cdot t-DC^T\xi_1\cdot\xi_1+C^T\xi_2\cdot t-DC^T\xi_2\cdot\xi_1).
    \end{align*}
    Consequently, 
    \begin{equation}\label{info1}\begin{split}
         \hat Sf(\xi)&=\mu_S\underbrace{e^{i\pi((-DC^T\xi_1+DB^+AD^T\xi_1)\cdot\xi_1+DC^T\xi_2\cdot\xi_2)}}_{(4)}\\
         &\times \int_{\ker(B)^\perp}f(t+D^T\xi_2)e^{i\pi B^+At\cdot t}\underbrace{e^{i\pi((-(DB^+A)^T\xi_1-B^+AD^T\xi_1+2C^T\xi_1)\cdot t}}_{(5)} e^{2\pi iC^T\xi_2\cdot t}dt.
  \end{split}  \end{equation}
		A straightforward calculation, using the symplectic relation $AD^T-BC^T=I_{d}$ in \eqref{sympRel-1}, the commutativity between the Moore-Penrose inverse and transposition, and Corollary \ref{corB2} $(i)$, allows us to rewrite the exponent in (4) as
		\begin{equation}\label{info2}
			-DC^T\xi_1\cdot \xi_1 +DB^+AD^T\xi_1\cdot\xi_1+DC^T\xi_2\cdot\xi_2
			=DB^+\xi_1\cdot\xi_1+DC^T\xi_2\cdot\xi_2.
		\end{equation}
		Next, let us focus on the exponent corresponding to the cross term (5). Applying the symplectic relation $AB^T=BA^T$ in \eqref{sympRel-1} to $x=(B^T)^+t$, where $t\in\ker(B)^\perp$ as above, we find:
		\begin{equation}\label{interpeq}
			At=BA^T(B^+)^Tt, \qquad t\in\ker(B)^\perp.
		\end{equation}
		Moreover, using \eqref{interpeq}, and observing that 
		$D^T\xi_1\in\ker(B)^\perp$ by Corollary \ref{corB2} $(i)$, we find: 
		\begin{align*}
			A^T(B^T)^+D^T\xi_1\cdot t+B^+AD^T\xi_1\cdot t-2C^T\xi_1\cdot t
			=2B^+AD^T\xi_1\cdot t-2C^T\xi_1\cdot t.
		\end{align*}		
		Using the relation $AD^T-BC^T=I_{d}$, we conclude that:
		\begin{equation}\label{info3}
		A^T(B^T)^+D^T\xi_1\cdot t+B^+AD^T\xi_1\cdot t-2C^T\xi_1\cdot t
			=2B^+\xi_1\cdot t.
		\end{equation}
		Formula \eqref{f2} follows by plugging \eqref{info2} and \eqref{info3} into \eqref{info1}.
		
\end{proof}

The following lemma  justifies the  decomposition  $\rd=R(B)\oplus A(\ker(B))$ in the above corollary.
\begin{lemma}\label{lemma46}
	Let $S\in \Sp(d,\bR)$ with blocks \eqref{blockS}. Then, $\rd=R(B)\oplus A(\ker(B))$. 
\end{lemma}
\begin{proof}
	First, we show that $R(B)+A(\ker(B))=\rd$, then we prove that the dimensions of the two spaces sum up to $d$. Let us denote by $S_1$ and $S_2$ the components of $S$:
	\[
		S\begin{pmatrix}
			x\\
			\xi
		\end{pmatrix}=\begin{pmatrix}
		A& B\\
		C & D
		\end{pmatrix}\begin{pmatrix}
			x\\
			\xi
		\end{pmatrix}
		=\begin{pmatrix}
			Ax+B\xi\\
			Cx+D\xi
		\end{pmatrix}=\begin{pmatrix}
			S_1(x,\xi)\\
			S_2(x,\xi)
		\end{pmatrix}.
	\]
	Since $S$ is surjective, $R(S_1)=\rd$. For $x\in\rd$, we write $x=x_1+x_2$ with $x_1\in \ker(B)^\perp$ and $x_2\in \ker(B)$. As a consequence of Corollary \ref{corB2}, illustrated in Table \ref{table:1} $(v)$, we have $A(\ker(B)^\perp)\subseteq R(B)$ and, consequently, for every $x,\xi\in\rd$,
	\[
		S_1(x,\xi)=Ax_1+Ax_2+B\xi\in R(B)+A(\ker(B)),
	\]
	It follows that $R(S_1)=\rd\subseteq R(B)+A(\ker(B))$. It remains to show that $$\mbox{dim}(A(\ker(B)))+\mbox{dim}(R(B))=d.$$ 
	If we write $\mbox{dim}(R(B))=r$, we only need to check that $$\mbox{dim}(A(\ker(B)))=d-r.$$ 
	Since $\mbox{dim}(\ker(B))=d-r$, it is enough to prove that the restriction of $A$ to $\ker(B)$ is injective. Let $x_1,x_2\in\ker(B)$, satisfying $Ax_1=Ax_2$.
	\[
		Ax_1=Ax_2\Leftrightarrow x_1-x_2\in\ker(A)\subseteq\ker(B)^\perp,
	\]
	by the straightforward consequence of Corollary \ref{corB2} reported in Table \ref{table:1} $(vii)$. Consequently, $x_1-x_2\in\ker(B)\cap\ker(B)^\perp=\{0\}$, and we are done.	
	
\end{proof}

\begin{example}
	In general $A(\ker(B))\not\perp R(B)$, as the following example illustrates. Let
	\[
		P=\begin{pmatrix}
			 0 & 0\\ 0 & 1
		\end{pmatrix}, \qquad E=\begin{pmatrix}
			-1 & 2\\
			-1 & 1
		\end{pmatrix}.
	\]
	Then, the symplectic matrix
	\[
		S=\mathcal{D}_EV_P^T=\begin{pmatrix} E^{-1} & E^{-1}P\\
		0_{2\times 2} & E^T
		 \end{pmatrix}=\left(\begin{array}{cc|cc}
		 	1 & -2 & 0 & -2\\
			1 & -1 & 0 & -1\\
			\hline
			0 & 0 & -1 & -1\\
			0 & 0 & 2 & 1
		 \end{array}\right),
	\]
	where $\mathcal{D}_E$ and $V_P$ are defined as in \eqref{defDL} and \eqref{defUQ} respectively, has
	\[
		\ker(B)=\ker(P)=\mbox{span}\left\{ \begin{pmatrix}1 \\ 0\end{pmatrix} \right\} \quad \Rightarrow \quad A(\ker(B))=\mbox{span}\left\{ \begin{pmatrix} 1 \\ 1 \end{pmatrix} \right\},
	\]
	whereas:
	\[
		R(B)=\mbox{span}\left\{\begin{pmatrix}2 \\ 1\end{pmatrix}\right\} \quad \Rightarrow \quad R(B)^\perp=\mbox{span}\left\{ \begin{pmatrix} -1 \\ 2 \end{pmatrix} \right\}.
	\]
	Clearly, even if $\bR^2=R(B)\oplus A(\ker(B))$ as expected, $R(B)^\perp\cap A(\ker(B))=\{0\}$.
	
\end{example}

\begin{lemma}\label{lemma46a}
	Let $S\in \Sp(d,\bR)$ with blocks \eqref{blockS}. Then, $\rd=\ker(B)^\perp\oplus D^TA(\ker(B))$. 
\end{lemma}

\begin{proof} It is clear that the sum of the dimensions of the subspaces is less than or equal to $\dim(\ker(B))^\perp+\dim(\ker(B))=d$. It suffices to show that for each $x\in \rd$ there is a decomposition $x=x_1+x_2$ such that $x_1\in(\ker(B))^\perp$ and $x_2\in D^TA(\ker(B))$. Moreover, clearly, it is enough to construct such decompositions for $x\in \ker(B)$. But by \eqref{sympRel}, we have $I_d=D^TA-B^TC$ and $x=-B^TCx+D^TAx$, where $-B^TCx\in R(B^T)=(\ker(B))^\perp$ and $D^TAx\in D^TA(\ker(B))$.

\end{proof}

To set up the Euclidean framework and compute the integrals over $\ker(B)^\perp$ above, we restate Corollary \ref{cor44} using a linear parametrization of $\ker(B)^\perp$. We fix an arbitrary  parametrization $V:\bR^r\to\ker(B)^\perp$ mapping the canonical basis of $\bR^r$ to an orthonormal basis of $\ker(B)^\perp$. In this case, the Moore-Penrose inverse of $V$ coincides with $V^T$, and $V^TV=I_{ r}$.

\[
\begin{tikzcd}[row sep=huge, column sep=huge]
    \ker(B)^\perp\subseteq\rd\arrow[r,"V^T", bend right=30, dashrightarrow] \arrow[r,"V", bend left=30,leftarrow]  \arrow[d, bend left = 30,"B"] & \bR^r  \arrow[dl, "BV",bend left=30]  \\
     R(B)\subseteq\rd \arrow[u,bend left=30,"B^+",dashrightarrow]
\end{tikzcd}
\]

\begin{corollary}\label{cor45}
	Under the notation of Corollary \ref{cor44}, 
	\[
		|\hat Sf(\xi)|=\mu_S |\hat g(V^TB^+\xi_1)|,
	\]
	where $\xi=\xi_1+\xi_2$, $\xi_1\in R(B)$ and $\xi_2\in A(\ker(B))$, and
	$\hat g$ is the Fourier transform on $\bR^r$ of the function
	\begin{equation}\label{cor45defg}
		g(u)=f(Vu+D^T\xi_2)e^{i\pi (V^TB^+AVu\cdot u-2V^TC^T\xi_2\cdot u)}, \qquad u\in\bR^r.
	\end{equation}
\end{corollary}

\begin{proof}[Proof of Corollary \ref{cor45}]
	It is a straightforward restatement of formula \eqref{f2}, using the change of variables defined through $V$. In fact,
	\begin{align*}
	\hat Sf(\xi)&\underset{\eqref{f2}}{=}\mu_S e^{i\pi (DB^+ \xi_1\cdot\xi_1+DC^T\xi_2\cdot\xi_2)}\int_{\ker(B)^\perp}f(t+D^T\xi_2)e^{i\pi B^+At\cdot t}e^{-2\pi i (B^+\xi_1+C^T\xi_2)\cdot t}dt\\
	&=\mu_S e^{i\pi (DB^+ \xi_1\cdot\xi_1+DC^T\xi_2\cdot\xi_2)}\int_{\bR^r}f(Vu+D^T\xi_2)e^{i\pi V^TB^+AVu\cdot u}e^{-2\pi i(V^TB^+\xi_1+V^TC^T\xi_2)\cdot t}dt\\
	&=\mu_S e^{i\pi (DB^+ \xi_1\cdot\xi_1+DC^T\xi_2\cdot\xi_2)}\int_{\bR^r}g(u)e^{-2\pi i(V^TB^+\xi_1)\cdot t}dt,
	\end{align*}
	where $g$ is defined as in \eqref{cor45defg}. This concludes the proof.
	
\end{proof}

\section{The uncertainty principle}\label{sec:HUP}
In this section, we prove Hardy's uncertainty principle for general metaplectic operators, regardless of the invertibility of the block $B$ of the corresponding projection. 
\subsection{Formulation of the main result} 
In view of the following remark, we may exclude the case $B\neq0_{d}$ in our analysis.
\begin{remark} The case $B=0_{d}$ is of little interest, trivial in many respects. Indeed, if $B=0_{d}$ in \eqref{blockS}, then $D=A^{-T}$ by \eqref{sympRel}, and
	\[
	S=\begin{pmatrix}
		A & 0_{d}\\
		C & A^{-T}\end{pmatrix}=\begin{pmatrix}
		A & 0_{d}\\
		0_{d} & A^{-T}
	\end{pmatrix}\begin{pmatrix}
		I_{d} & 0_{d}\\
		A^TC & I_{d}
	\end{pmatrix}.
	\]
	Consequently, up to a sign,
	\[
	\hat Sf(\xi)=|\det(A)|^{-1/2}e^{-i\pi CA^{-1}\xi\cdot\xi}f(A^{-1}\xi), \qquad \xi\in\bR^d,
	\]
	which is basically a rescaling. 
\end{remark}

%
In stating our generalization of Theorem \ref{thmIntro4}, we use the decompositions of $\rd$ obtained in Lemma \ref{lemma46} and \ref{lemma46a}.

\begin{theorem}\label{Hardyglobale}
	Let $\hat S\in\Mp(d,\bR)$ be a metaplectic operator with projection $S=\pi^{Mp}(\hat S)$ having block decomposition \eqref{blockS}. Let $1\leq r=rank(B)\leq d$. Consider $M,N\in\Sym_+(d,\bR)$  satisfying 
	\begin{align}
		\label{G1}
		&\ker(M)=\ker(B),\\
		\label{G2}
		&R(N)=R(B).
	\end{align}
	Let $f\in L^2(\rd)\setminus\{0\}$. Assume that for almost every $x_2\in D^TA(\ker(B))$ and almost every $\xi_2\in A(\ker(B))$ (with respect to the Lebesgue measure on $\bR^r$), $f$ satisfies the decay estimates
	\begin{align}
	\label{G1-2}
	& |f(x_1+x_2)|\leq \alpha(x_2)e^{-\pi Mx_1\cdot x_1}, \qquad x_1\in \ker(B)^\perp,\\
	\label{G2-2}
	& |\hat Sf(\xi_1+\xi_2)|\leq \beta(\xi_2)e^{-\pi N\xi_1\cdot\xi_1}, \qquad \xi_1\in R(B),
\end{align}
where $\alpha$ and $\beta$ are measurable almost everywhere finite functions.
	Let $\lambda_1,\ldots,\lambda_r$ be the non-zero eigenvalues of $MB^TNB$. Then,
\begin{itemize}
	\item[$(i)$] $\lambda_j\leq 1$ for every $j=1,\ldots,r$.
	\item[$(ii)$] If $\lambda_1=\ldots=\lambda_r=1$, then, 
	\begin{align}\label{Th2}
	f(x_1+x_2)=\gamma(x)e^{-\pi(M+iB^+A)x_1\cdot x_1}e^{2\pi iC^TAx\cdot x_1}, \qquad x_1\in\ker(B)^\perp,
	\end{align} 
	for some $\gamma\in L^2(\ker(B))$, where $x_2=D^TAx$ and $x\in\ker(B)$.
\end{itemize}

\end{theorem}

As a corollary to Theorem \ref{Hardyglobale}, we obtain:
\begin{corollary}
	Under the assumptions of Theorem \ref{Hardyglobale}, if $\lambda_j>1$ for some $j=1,\ldots,r$, then $f=0$ almost everywhere.
\end{corollary}

Let us comment on the statement of Theorem \ref{Hardyglobale}.

\begin{remark}
The matrices $M$ and $N$, which are allowed to be positive-semidefinite, track the decay of $f$ and $\hat Sf$ along the directions of $\ker(B)^\perp$ and $R(B)$, respectively. This is encoded in \eqref{G1} and \eqref{G2}, which have fundamental consequences also on the properties of the eigenvalues of $MB^TNB$, as outlined in Lemma \ref{lemma52}.
\end{remark}


\begin{remark}
    The spaces $D^TA(\ker(B))$ and $A(\ker(B))$ can be replaced with any pair of subspaces $\mathcal{L}_1,\mathcal{L}_2\subseteq\bR^d$
    complementing $\ker(B)^\perp$ and $R(B)$. We will return to this fact in Subsection \ref{sec:sharpness}, where we will prove that the directions outside $\ker(B)^\perp$ and $R(B)$ cannot contribute actively to any uncertainty principle for metaplectic operators, see Theorem \ref{sharp} below. 
\end{remark}

As aforementioned, Theorem \ref{Hardyglobale} was proved by Dias, de Gosson and Prata for metaplectic operators with free projections, i.e., having $B\in\GL(d,\bR)$, cf. \cite[Theorem 29]{CGP2024}. In what follows, we only need this result for the Fourier transform.

\begin{theorem}\label{thmHardyAnis}
	Let $M,N\in\Sym_{++}(r,\bR)$ and $g\in L^2(\bR^r)\setminus\{0\}$ be such that:
	\begin{align*}
		|g(u)|&\lesssim e^{-\pi Mu\cdot u}, \qquad u\in\bR^r,\\
		|\hat g(\eta)|&\lesssim e^{-\pi N\eta\cdot\eta}, \qquad\eta\in\bR^r.
	\end{align*}
	Let $\lambda_1,\ldots,\lambda_r$ denote the eigenvalues of $MN$. Then, the following theses hold true:
	\begin{itemize}
		\item[$(i)$] $\lambda_1,\ldots,\lambda_r\leq1$.
		\item[$(ii)$]  If $\lambda_1=\ldots=\lambda_r=1$, then 
		\[
		g(u)=Ce^{-\pi Mu\cdot u}, \qquad u\in\bR^r,
		\]
        for some $C\in\bC$.
	\end{itemize}
\end{theorem}




\subsection{Proof of the main result}
We use the results of Section \ref{sec:integRep} to deduce  Theorem \ref{Hardyglobale} from  Theorem \ref{thmHardyAnis}. We claim that for almost every $x\in \ker(B)$ the following inequalities hold:
\[|f(x_1+D^TAx)| \le \tilde{\alpha}(x)e^{-\pi Mx_1\cdot x_1},\quad |\hat S f(\xi_1+Ax)|\le \tilde{\beta}(x)e^{-\pi N\xi_1\cdot \xi_1},\]
where $x_1\in(\ker(B))^\perp$ and $\xi_1\in R(B)$.
Let $g_x(u)$ be given by \eqref{cor45defg} for $\xi_2=Ax$. Then, by Corollary \ref{cor45}, we have
\[|\hat S f(\xi_1+Ax)|=\mu_S|\hat g_x(V^TB^+\xi_1)|.\]
We note that $B^+$ is the inverse of $B$ on $R(B)$. Reformulating the decay conditions in terms of $g_x$, we conclude that for almost all $x\in\ker(B)$,
\begin{align*}
&|g_x(u)|\le \tilde{\alpha}(x)e^{-\pi V^TMVu\cdot u}, \qquad u\in\bR^r\\
&|\hat{g}_x(\eta)|\le \tilde{\beta}(x)e^{-\pi V^TB^TNBV\eta\cdot \eta}, \qquad \eta\in\bR^r. 
\end{align*}
We wish to apply Theorem \ref{thmHardyAnis} to $g_x$ with matrices $V^TMV$ and $V^TB^TNBV$. First, we show that\\
	(a) $V^TMV,V^TB^TNBV\in\Sym_{++}(r,\bR)$,\\
	(b) $V^TMVV^TB^TNBV=V^TMB^TNBV$.\\
	\textbf{Proof of $(a)$.} Since $M$ and $N$ are positive-semidefinite, the products $V^TMV$ and $(BV)^TNBV$ are also positive-semidefinite. Moreover, since $V$ and $BV$ are parametrizations of $\ker(B)^\perp$ and $R(B)$, respectively, and $\ker(M)^\perp=\ker(B)^\perp$ and $R(B)=R(N)$, we have that $V^TMV$ and $V^TB^TNBV=(BV)^TN(BV)$ define isomorphism of $\bR^r$ and, consequently, they are positive-definite matrices. \\
%
	\textbf{Proof of $(b)$.} This follows trivially observing that $VV^+=VV^T$ is an isomorphism on $\ker(B)^\perp$ and that for every $u\in\bR^r$, $B^TNBVu\in \ker(B)^\perp$, so that 
	\[
		(VV^T)B^TNBVu=B^TNBVu.
	\]
	This proves $(b)$.\\

Applying Theorem \ref{thmHardyAnis} for $x$ in a set of full measure, we obtain that if $\lambda_1,\ldots, \lambda_r$ are the eigenvalues of $V^TMB^TNBV$, then
\begin{itemize}
\item[$(i)$] $\lambda_1,\ldots, \lambda_r\le 1$.
\item[$(ii)$] if $\lambda_1=\ldots=\lambda_r=1$, then $g_x(u)=\gamma(x)e^{-\pi V^TMVu\cdot u}$.
\end{itemize}
To finish the proof of Theorem \ref{Hardyglobale}, we note that the eigenvalues of $V^TMB^TNBV$ coincide with non-zero eigenvalues of $MB^TNB$, see Appendix C, and that for the case $(ii)$ $\lambda_1=...=\lambda_r=1$, we get
\[f(x_1+D^TAx)=\gamma(x)e^{-\pi Mx_1\cdot x_1}e^{-i\pi (B^+Ax_1\cdot x_1-2C^TAx\cdot x_1)}.\]
Since $f\in L^2(\R^d)$, we conclude that $\gamma\in L^2(\ker(B))$.

\if false

\begin{itemize}
	\item[Step 1.] We re-write decay conditions for $f|_{\ker(B)^\perp}$ and $(\hat Sf)|_{R(B)}$ in terms of decay conditions for $f_e$ and $\hat f_e$ on $\bR^r$, where $f_e$ is an Euclidean projection of $f$.
	\item[Step 2.] We prove Hardy's uncertainty principle for $f|_{\ker(B)^\perp}$.
	\item[Step 3.] Step 2, along with some consideration on the block $A$ in \eqref{blockS}, allows to conclude the proof of Theorem \ref{Hardyglobale}.
	
\end{itemize}

\subsection{Step 1: from the Euclidean projection back to $f|_{\ker(B)^\perp}$} 

In Section \ref{sec:integRep}, we parametrized $\ker(B)^\perp$ to obtain Euclidean projections of a signal $f$, to which we will apply Theorem \ref{thmHardyAnis}. For, we need to relate the decay conditions of $f$ with those of $f_e$.  
\begin{theorem}\label{thm51}
	Let $f\in L^2(\rd)$ and $f_e$ be the Euclidean projection with respect to the parametrization $V$. The inequalities:
	\begin{align}
	\label{boundf}
		&|f(x)|\lesssim e^{-\pi Mx\cdot x}, \qquad x\in\ker(B)^\perp\\
	\label{boundSf}
		&|\hat Sf(\xi)|\lesssim e^{-\pi N\xi\cdot\xi}, \qquad \xi \in R(B),
	\end{align}
    hold if and only if the inequalities
	\begin{align}
		\label{bg}
		&|f_{e}(u)|\lesssim  e^{-\pi  V^TMVu\cdot u}, \qquad u\in\bR^r\\
		\label{bSgp}
		&|\hat f_{e}(\eta)|\lesssim e^{-\pi V^TB^TNBV\eta\cdot\eta}, \qquad \eta\in\bR^r.
	\end{align}
\end{theorem}
\begin{proof}
	We show that \eqref{boundf} and $\eqref{boundSf}$ imply \eqref{bg} and \eqref{bSgp}, the converse is proved in a similar fashion. Clearly, for every $u\in\bR^r$,
	\begin{equation}\label{bSg}
		|f_{e}(u)|\underset{\eqref{mod-g}}{=}|f(Vu)|\underset{\eqref{boundf}}{\lesssim} e^{-\pi M(Vu)(Vu)}=e^{-\pi  V^TMVu\cdot u}, \qquad u\in\bR^r.
	\end{equation}
	Similarly, for every $\eta\in\bR^r$,
	\begin{equation}\label{appoggio2}
		|\hat f_{e}(\eta)|\underset{\eqref{mod-Fg}}{=}\frac{1}{\mu_{S}} |\hat Sf(BV\eta)|\underset{\eqref{boundSf}}{\lesssim} e^{-\pi V^TB^TNBV\eta\cdot \eta}, \qquad \eta\in\bR^r.
	\end{equation}
	
\end{proof}

Let us mention that the decay conditions \eqref{boundf} and \eqref{boundSf} imply the regularity of $f|_{\ker(B)^\perp}$ and $(\hat Sf)|_{R(B)}$, as observed hereafter.

\begin{corollary}\label{corReg}
	Let $f\in L^2(\rd)$ satisfy \eqref{boundf} and \eqref{boundSf}. Then, $f|_{\ker(B)^\perp}\in\mathcal{C}^\infty(\ker(B)^\perp)$ and $(\hat Sf)|_{R(B)}\in\mathcal{C}^\infty(R(B))$.
\end{corollary}
\begin{proof}
	
	 The directional decay estimates \eqref{boundf} and \eqref{boundSf} correspond to the global Gaussian decay estimates \eqref{bg} and \eqref{bSgp} for $f_{e}$. The classical theory of the Fourier transform guarantees that both $f_{e}$ and $\hat f_{e}$ are smooth functions on $\bR^r$. Consequently, $f|_{\ker(B)^\perp}$ and $\hat Sf|_{R(B)}$ are also smooth, by \eqref{f3} and \eqref{cor45defg}. 
	
\end{proof}

\subsection{ Step 2: the uncertainty principle for $f|_{\ker(B)^\perp}$ }

At this point of the proof, we are able to handle the restrictions of signals to 
$\ker(B)^\perp$, and therefore we need a version of Theorem \ref{thmHardyAnis} in this non-Euclidean setting.
First, we need to check that the non-zero (positive) eigenvalues of $MB^TNB$ are related to those of the corresponding matrix from the Euclidean perspective, which is an $r\times r$ matrix. This is a consequence of the conditions \eqref{G1} and \eqref{G2} on $M$ and $N$. The  proof of the following lemma is tucked away in the appendix.

We are ready to state and prove Hardy's uncertainty principle for $f|_{\ker(B)^\perp}$.

\begin{proposition}\label{propIntro4}
	Let $M,N\in\Sym_+(d,\bR)$ satisfying \eqref{G1} and \eqref{G2}. Assume that $f\in L^2(\rd)$ has $f|_{\ker(B)^\perp}\not\equiv0$, and satifies \eqref{boundf} and \eqref{boundSf}. 
	Let $\lambda_1,\ldots,\lambda_r$ denote the non-zero eigenvalues of $MB^TNB$. 
	Then the following theses hold true:
	\begin{itemize}
		\item[$(i)$] $\lambda_j\leq 1$ for every $j=1,\ldots,r$.
		\item[$(ii)$]  If $\lambda_1=\ldots=\lambda_r=1$, then 
		\[
		f(x)=C e^{-\pi (M+iB^+A)x\cdot x}, \qquad x\in \ker(B)^\perp,
		\]
		for some $C\in\bC$.		
	\end{itemize}
\end{proposition}

\begin{proof}
	Consider an Euclidean projection $f_e$ of $f$. By Lemma \ref{lemma52}, $\{\lambda_1,\ldots,\lambda_r\}\subset(0,+\infty)$ are also the eigenvalues of $V^TMB^TNBV$.
	
	By Theorem \ref{thm51}, the directional decay estimates \eqref{G1} and \eqref{G2} for $f$, read as the global decay estimates \eqref{bg} and \eqref{bSgp} for $f_e$. 	
	Then, the assertion is a straightforward restatement of Theorem \ref{thm51} applied to $f_e$, which satisfies the estimates
	\begin{align*}
		|f_e(u)|&\lesssim e^{-\pi M'u\cdot u}, \qquad u\in\bR^r,\\
		|\hat f_e(\eta)|&\lesssim e^{-\pi N'\eta\cdot\eta}, \qquad \eta\in\bR^r,
	\end{align*}
	with $M'=V^TMV$ and $N'=V^TB^TNBV$.
	
\end{proof}

We have a corresponding non-Euclidean version of Corollay \ref{E1}.

\begin{corollary}\label{corPWG}
	Under the assumptions of Proposition \ref{propIntro4}, if $\lambda_j>1$ for some $j=1,\ldots,r$, $f|_{\ker(B)^\perp}\equiv0$.
\end{corollary}
\begin{proof}
	Let $f_e$ be an Euclidean projection of $f$. If $\lambda_j>1$ for some $j=1,\ldots,d$, $f_e(u)=0$ for every $u\in\bR^r$ by classical Hardy's uncertainty principle. Consequently, $|f(Vu)|=|f_e(u)|=0$ for every $u\in\bR^r$. Since $V$ is a parametrization of $\ker(B)^\perp$, this is equivalent to claim that $f|_{\ker(B)^\perp}\equiv0$.
	
\end{proof}

\begin{remark}
    Corollary \ref{corPWG} is used implicitly in Theorem \ref{Hardyglobale}, since it guarantees the pointwise validity of \eqref{G1-2} and \eqref{G2-2} for $x_1\in\ker(B)^\perp$ and $\xi_1\in R(B)$.
\end{remark}


\subsection{From $\ker(B)^\perp$ to $\rd$}
Before proving Hardy's uncertainty principle, we show that $\rd$ can be decomposed as the direct sum $R(B)\oplus A(\ker(B))$.

In what follows, if $x,\xi\in\rd$, we shall always write
\[
	x=x_1+x_2, \qquad x_1\in\ker(B)^\perp, x_2\in\ker(B),
\]
and
\[
	\xi=\xi_1+\xi_2, \qquad \xi_1\in R(B), \xi_2\in A(\ker(B)).
\]

\begin{remark}\label{rm49}
	It is worth noting that from the proof of the previous theorem, we can extrapolate that $A:\ker(B)\to A(\ker(B))$ is an isomorphism. Applying this result to $S^{-1}$, with block decomposition \eqref{blockS-1}, we also have that $D^T:R(B)^\perp\to D^T(R(B)^\perp)$ is an isomorphism.
\end{remark}

Under the assumptions \eqref{G1-2} and \eqref{G2-2} of Theorem \ref{Hardyglobale}, we can prove that for almost every $x_2\in\ker(B)$, the function $\tilde f(x_1)=f(x_1+x_2)$, $x_1\in\ker(B)^\perp$ is smooth.

\begin{lemma}\label{lemma412}
Under the assumptions of Theorem \ref{Hardyglobale}, we have that $f|_{x_2+\ker(B)^\perp}\in\mathcal{C}^\infty(\ker(B)^\perp)$ for almost every $x_2\in\ker(B)$, and $\hat Sf|_{\xi_2+R(B)}\in\mathcal{C}^\infty(R(B))$ for almost every $\xi_2\in A(\ker(B))$. 
\end{lemma}
\begin{proof}	
	Let $x_2\in \ker(B)$ so that $f(x_2)$ is defined. Consider the auxiliary function $\tilde f\in L^2(\rd)$ defined as $\tilde f(x)=f(x+x_2)=T_{-x_2}f(x)$, where $T_{-x_2}=\rho(-x_2,0;0)$ is the translation operator. We have:
	\begin{align*}
		|\tilde f(x_1)|=|f(x_1+x_2)|\underset{\eqref{G1-2}}{\leq} \alpha(x_2)e^{-\pi Mx_1\cdot x_1}, \qquad x_1\in\ker(B)^\perp.
	\end{align*}
	Moreover, for every $ \xi_1\in R(B)$,
	\begin{equation}\begin{split}\label{pe413}
		\hat S\tilde f(\xi_1)&=\hat S\rho(-x_2,0;0)f(\xi_1)\underset{\eqref{intertS}}{=}\rho(S(-x_2,0;0))\hat Sf(\xi_1)\\
		&=\rho(-Ax_2,-Cx_2;0)\hat Sf(\xi_1)=e^{-2\pi iCx_2\cdot \xi_1}\hat Sf(\xi_1+Ax_2).
	\end{split}\end{equation}
	Consequently,
	\[
		|\hat S\tilde f(\xi_1)|\underset{\eqref{G2-2}}{\leq}\beta(Ax_2)e^{-\pi N\xi_1\cdot\xi_1}.
	\]
	The assumptions of Corollary \ref{corReg} are fulfilled and, therefore, $\tilde f|_{\ker(B)^\perp}\in \mathcal{C}^\infty(\ker(B)^\perp)$ and $\hat S\tilde f|_{R(B)}\in\mathcal{C}^\infty(R(B))$. For $\tilde f$, this means that $f|_{x_2+\ker(B)^\perp}\in\mathcal{C}^\infty(\ker(B)^\perp)$ for every $x_2\in\ker(B)$ such that $f(x_2)$ is defined. Moreover, it follows by \eqref{pe413} that $\hat S\tilde f|_{R(B)}\in\mathcal{C}^\infty(R(B))$. We conclude that $\hat Sf(\cdot+Ax_2)\in\mathcal{C}^\infty(R(B))$ for almost every $x_2\in \ker(A)$ or, equivalently, $\hat Sf\in\mathcal{C}^\infty(R(B))$ for almost every $\xi_2\in A(\ker(B))$. This concludes the proof.

	\end{proof}
	
	\begin{example}
	Let us stress that the directionality in Corollary \ref{lemma412} is unavoidable. For instance, let $\hat S=\cF_2$ be the partial Fourier transform with respect to the second variable, defined by $\cF_2F(x,\xi)=\int_{-\infty}^\infty F(x,t)e^{-2\pi it\xi}dt$, $F\in\cS(\bR^2)$, and with symplectic projection given by the symplectic interchange
	\[
		\Pi_2=\left(\begin{array}{cc|cc}
			1 & 0 & 0 & 0\\
			0 & 0 & 0 & 1\\
			\hline
			0 & 0 & 1 & 0\\
			0 & -1 & 0 & 0
		\end{array}\right).
	\]
	We have:
	\[
		\ker(B)^\perp=R(B)=\mbox{span}\left\{ \begin{pmatrix} 0 \\ 1\end{pmatrix}\right\}.
	\]
	Let $f(x,y)=f_1(x)f_2(y)$, $f_1\in L^2(\bR)$, with $f_1(0)\neq0$, and $f_2(y)=e^{-\pi y^2}$. Then,
	\[
		|f(0,y)|= |f_1(0)|e^{-\pi y^2}, \qquad and \qquad |\cF_2f(0,\xi)|=|f_1(0)|e^{-\pi y^2}.
	\]
	Hence, $f$ satisfies the assumptions of Corollary \ref{corReg} and indeed $f_{(x,0)+\ker(B)^\perp}$ features regularity only on $(x,0)+\ker(B)^\perp$ for those $x\in\bR$ so that $f_1(x)$ is defined. 
\end{example}

Having proved the many lemmas and proposition above, the conclusion of the proof of Theorem \ref{Hardyglobale} is now a matter of a few lines.

%
%
%
\begin{proof}
	Fix $x_2\in\ker(B)$ so that $f(x_2)$ is defined. Consider again $\tilde f=T_{-x_2}f$. In the proof of Lemma \ref{lemma412}, we have obtained the inequalities:
	\begin{align}\label{GE1}
		|\tilde f(x_1)|\leq \alpha(x_2)e^{-\pi Mx_1\cdot x_1}, \qquad x_1\in\ker(B)^\perp.
	\end{align}
	and
	\begin{align}\label{GE2}
		|\hat S\tilde f(\xi_1)|&\leq\beta(Ax_2)e^{-\pi N\xi_1\cdot\xi_1}, \qquad \xi_1\in R(B),
	\end{align}
	holding for almost every $x_2\in\ker(A)$, specifically when $f(x_2)$ is defined.
	
	Now, $(i)$ is exactly $(i)$ of Proposition \ref{propIntro4}. Moreover, by item $(ii)$ of Proposition \ref{propIntro4}, if $\lambda_1=\ldots=\lambda_r=1$, then there exists $C=C(x_2)>0$ such that:
	\[	
		\tilde f(x_1)=Ce^{-\pi (M+iB^+A)x_1\cdot x_1}, \qquad x_1\in\ker(B)^\perp,
	\]
	which implies $(ii)$, since by Tonelli's theorem:
	\begin{align*}
		\norm{f}_2^2&=\int_{\rd}|C(x_2)|^2 e^{-2\pi Mx_1\cdot x_1}dx_1dx_2=\\
		&\norm{C}_{L^2(\ker(B))}^2\int_{\ker(B)^\perp}e^{-2\pi Mx_1\cdot x_1}dx_1\asymp\norm{C}_{L^2(\ker(B))}^2.
	\end{align*}
	
\end{proof}

The proof of Hardy's uncertainty principle is done. Again, we get a generalization of Corollary \ref{E1}.

\begin{corollary}
	Under the assumptions of Theorem \ref{Hardyglobale}, if $f\in L^2(\rd)$ and $\lambda_j>1$ for some $j=1,\ldots,d$, then $f=0$ almost everywhere.
\end{corollary}
\begin{proof}
	By Theorem \ref{Hardyglobale}, $f|_{x_2+\ker(B)^\perp}\equiv0$ for almost every $x_2\in\ker(B)$, that is the assertion.
	
\end{proof}

%
\fi
\subsection{Considerations about the decay directions}\label{sec:sharpness}
A straightforward argument shows that the decompositions of $\rd$:
	\begin{equation}\label{cbr}
		\rd=\ker(B)^\perp\oplus D^TA(\ker(B))=R(B)\oplus A(\ker(B)),
	\end{equation}
	that simplified the proof of Hardy's uncertainty principle, can be relaxed by replacing $D^TA(\ker(B))$ and $A(\ker(B))$ with any pair of $(d-r)$-dimensional subspaces $\mathcal{L}_1,\mathcal{L}_2\subseteq\rd$ such that the identities:
	\begin{equation}\label{asmpL1L2}
		\rd=\ker(B)^\perp\oplus \mathcal{L}_1=R(B)\oplus \mathcal{L}_2
	\end{equation}
	hold. Providing extensive details on this point would unnecessarily burden the discussion, so we limit it to a few lines to convince the reader of its validity. 
	\begin{itemize}
	\item Under the assumption \eqref{asmpL1L2}, Lemma \ref{lemma46} and Lemma \ref{lemma46a} hold for $\mathcal{L}_1$ in replacing $\ker(B)$ with the very same proofs.
	\item Observe that $\rd=\ker(B)^\perp\oplus D^TA(\ker(B))$ implies also that $\dim(D^TA(\ker(B))=\dim(\ker(B))$ and, consequently, $D^TA:\ker(B)\to D^TA(\ker(B))$ is an isomorphism. 
	\item The decay conditions \eqref{G1-2} and \eqref{G2-2} are invariant when applying a linear isomorphisms to $D^TA(\ker(B))$ and $A(\ker(B))$, up to modify $\alpha$ and $\beta$ accordingly.
	\end{itemize}
	
	However, the same cannot be inferred for $\ker(B)^\perp$ and $R(B)$, that cannot be replaced in \eqref{cbr} with any other $r$-dimensional subspaces of $\rd$.
	
	This flexibility in choosing the spaces complementing $\ker(B)^\perp$ and $R(B)$ in \eqref{G1} and \eqref{G2}, allows us decomposing:
	\begin{equation}\label{decompD}
	\rd=\ker(B)^\perp\oplus D^T(R(B)^\perp)=R(B)\oplus R(B)^\perp.
	\end{equation}

	\begin{remark}
	It follows by applying Lemma \ref{lemma46} to $S^{-1}$ that $D^T:R(B)^\perp\to D^T(R(B)^\perp)$ isomorphically, so if $K\subseteq D^T(R(B)^\perp)$ is compact, then $D^{-T}(K)$ is also compact in $R(B)^\perp$, where $D^{-T}$ is the inverse of $D^T|_{R(B)^\perp}$. 
	\end{remark}
	
	We are ready to establish the sharpness of \eqref{G1} and \eqref{G2}.

\begin{theorem}\label{sharp}
	Let $\hat S\in \Mp(d,\bR)$ have projection $S$ with block decomposition \eqref{blockS}. Let $1\leq r=rank(B)<d$. Then, for every compact $K\subseteq D^T(R(B)^\perp)$ there exists a function $f\in L^2(\rd)$ such that $\mbox{supp}(f)\subseteq \{x:x_2\in K\}$ and $\mbox{supp}(\hat Sf)\subseteq \{\xi:\xi_2\in D^{-T}(K)\}$.
\end{theorem}
\begin{proof}
	As observed above, let us decompose $\rd$ as in \eqref{decompD}. The integral representation \eqref{intrepformula} applied to a function that tensorizes as $f(x_1+x_2)=f_1(x_1)f_2(x_2)$, $x_1\in\ker(B)^\perp$, $x_2\in D^T(R(B)^\perp)$, gives:
	\begin{align*}
		\hat Sf(\xi_1+\xi_2)&=\mu_Se^{i\pi DC^T\xi\cdot\xi}\int_{\ker(B)^\perp}f(x+D^T\xi_1+D^T\xi_2)e^{i\pi B^+Ax\cdot x}e^{2\pi iC^T\xi\cdot x}dx\\
		&=\mu_Se^{i\pi DC^T\xi\cdot\xi}\int_{\ker(B)^\perp}f_1(x+D^T\xi_1)f_2(D^T\xi_2)e^{i\pi B^+Ax\cdot x}e^{2\pi iC^T\xi\cdot x}dx\\
		&=\mu_S{e^{i\pi DC^T\xi\cdot\xi}f_2(D^T\xi_2)\int_{\ker(B)^\perp}f_1(x+D^T\xi_1)e^{i\pi B^+Ax\cdot x}e^{2\pi iC^T\xi\cdot x}dx},\\
	\end{align*}
	where $\xi=\xi_1+\xi_2$, $\xi_1\in R(B)$, $\xi_2\in R(B)^\perp$. Consequently, $\hat Sf$ vanishes wherever $f_2\circ D^T$ vanishes. The assertion follows by choosing any function $f_1\in L^1(\ker(B)^\perp)\cap L^2(\ker(B)^\perp)$ and 
	\[
		f_2(x_2)=\chi_K(x_2)=\begin{cases}
			1 & \text{if $x_2\in K$},\\
			0 & \text{otherwise},
		\end{cases}
	\]
	 the characteristic function on $K$.
	
\end{proof}

\begin{remark}
    The importance of Theorem \ref{sharp} is not limited to Hardy's uncertainty principle. Indeed, it provides a \emph{negative} result for any uncertainty principle for metaplectic operators: the directions outside $\ker(B)^\perp$ and $R(B)$ do not contribute actively to uncertainty principles.
\end{remark}

\subsection{Examples for Hardy's uncertainty principle}

\subsubsection{Metaplectic multipliers}
A straightforward example concerns metaplectic multipliers, as defined in Example \ref{exMetap} (d). Let $P\in\Sym(d,\bR)$ and 
\[
	\mathfrak{m}_Pf=\cF^{-1}(\Phi_{-P}\hat f),
\]
where $\Phi_P(t)=e^{-i\pi Pt\cdot t}$ is a chirp. The associated projection is the matrix $U_P$ defined in \eqref{defUQ}, having:
\[
	A=I_{d}, \quad B=P, \quad C=0_{d} \quad \text{and} \quad D=I_d.
\]
Since $P$ is symmetric and $A$ is the identity, we have
\[
	R(B)=\ker(B)^\perp=\ker(P)^\perp,
\]
and $A(\ker(B))=\ker(P)$. Consequently, conditions \eqref{G1-2} and \eqref{G2-2} can be rephrased as:
\begin{align}
	\label{G1-3}
		&|f(x)|\leq \alpha(x_2)e^{-\pi Mx\cdot x}, \qquad a.e. \ x\in\rd,\\
	\label{G2-3}
		&|\hat Sf(\xi)|\leq \beta(\xi_2)e^{-\pi N\xi\cdot\xi}, \qquad a.e. \ \xi\in\rd,
\end{align}
where $x=x_1+x_2$ and $\xi=\xi_1+\xi_2$, as usual.
\begin{corollary}
	Let $\mathfrak{m}_P$ be defined as above, and $1\leq r=rank(P)\leq d$. Let $M,N\in\Sym_+(d,\bR)$ with
	\[
		\ker(M)=\ker(N)=\ker(P).
	\]
	Let $f\in L^2(\rd)\setminus\{0\}$ satisfy \eqref{G1-3} and \eqref{G2-3}, and $\lambda_1,\ldots,\lambda_r$ be the positive eigenvalues of $MPNP$. Then,\\
	(i) $\lambda_1,\ldots,\lambda_r\leq1$.\\
	(ii) If $\lambda_1=\ldots=\lambda_r=1$, then
		\[	
			f(x)=\gamma(x_2)e^{-\pi(M+iP^+)x\cdot x}, \qquad a.e. \ x\in\rd,
		\]
		where $\gamma\in L^2(\ker(P))$.
\end{corollary}

\subsubsection{Fractional Fourier transforms}
Let $\vartheta=(\vartheta_1,\ldots,\vartheta_d)\in\rd$. The \emph{fractional Fourier transform} $\cF_\vartheta$ is a metaplectic operator in $\Mp(d,\bR)$ with projection
\begin{equation}\label{matFRFT}
	S_\vartheta=\left(\begin{array}{ccc|ccc}
		\cos\vartheta_1 & \ldots & 0 & \sin\vartheta_1 & \ldots & 0\\
		\vdots & \ddots & \vdots & \vdots & \ddots & \vdots\\
		0 & \ldots & \cos\vartheta_d & 0 & \ldots & \sin\vartheta_d\\
		\hline
		-\sin\vartheta_1 & \ldots & 0 & \cos\vartheta_1 & \ldots & 0\\
		\vdots & \ddots & \vdots & \vdots & \ddots & \vdots\\
		0 & \ldots & -\sin\vartheta_d & 0 & \ldots & \cos\vartheta_d
	\end{array}\right)=\begin{pmatrix}A_\vartheta & B_\vartheta \\ C_\vartheta & D_\vartheta\end{pmatrix}.
\end{equation}
We stress that here we allow $\sin\vartheta_j=0$. If $B_\vartheta\in\GL(d,\bR)$, we retrieve the fractional Fourier transform in \cite{MO2002}, simply observing that in this case $R(B_\vartheta)=\ker(B_\vartheta)^\perp=\rd$, and writing the integral representation \eqref{f2} for \eqref{matFRFT}. 

Let $\cJ=\{j:\sin\vartheta_j\neq0\}$ and $\cJ^c=\{1,\ldots,d\}\setminus\cJ$. We have:
\[
	\ker(B_\vartheta)=\{x\in\rd : x_j=0, \ j\in\cJ\}=:\Gamma_{\cJ^c},
\]
and
\[
	 \ker(B_\vartheta)^\perp=R(B_\vartheta)=\{x:x_j=0, \ j\in\cJ^c\}=:\Gamma_{\cJ}.
\]
We assume that $\sin(\vartheta_j)\neq0$ for at least one $j\in\{1,\ldots,d\}$, so that $r=rank(B_\vartheta)\geq1$. If $M,N\in\Sym_+(d,\bR)$ are positive-semidefinite matrices with
	\[
		\ker(M)^\perp=R(N)=\Gamma_\cJ,
	\]
	we need to consider the positive eigenvalues of $MBNB$. In particular, if $M=aI_\cJ$ and $N=bI_\cJ$, where $a,b>0$ and $I_\cJ$ is the diagonal matrix with diagonal entries 
	\[
		(I_\cJ)_{jj}=\begin{cases}
			1 & \text{if $j\in\cJ$},\\
			0 & \text{otherwise},
		\end{cases}
	\]
	then the eigenvalues of $MBNB$ are $\lambda_j=0$ if $j\in\cJ^c$ and $\lambda_j=ab\sin^2\vartheta_j$ if $j\in\cJ$. Let us also denote $I_{\cJ^c}=I_{d}-I_{\cJ}$, and observe that the products $I_\cJ x$ and $I_{\cJ^c}x$ are the vectors of $\rd$ having their coordinates indexed by $\cJ^c$ and $\cJ$ replaced with $0$, respectively.
	
	Considering that $\ker(C)=\ker(B)$ in this context, and consequently $C^TAx=A^TCx=0$ for every $x\in\ker(B)$, Theorem \ref{Hardyglobale} for fractional Fourier transforms takes the following form.
\begin{corollary}
	Under the notation above, if $f\in L^2(\rd)\setminus\{0\}$ satisfies
	\begin{align*}
		|f(x)|\leq \alpha(I_{\cJ^c}x)e^{-\pi a|I_\cJ x|^2}, \qquad a.e. \ x\in\rd,\\
		|\hat Sf(\xi)|\leq\beta(I_{\cJ^c}\xi)e^{-\pi b|I_\cJ \xi|^2}, \qquad a.e. \ \xi\in\rd,\\
	\end{align*}
	for some $\alpha,\beta:\Gamma_{\cJ^c}\to (0,+\infty)$, then:
	\begin{itemize}
	\item[$(i)$] $ab\sin^2\vartheta_j\leq1$ for every $1\leq j\leq d$.
	\item[$(ii)$] If $\sin^2\vartheta_j=1/ab$ for every $1\leq j\leq d$, we have
		\[
			f(x)=\gamma(I_{\cJ^c}x)\prod_{j\in\cJ}e^{-i\pi (a+i\cot\vartheta_j)|x_j|^2}, \qquad a.e. \ x\in\rd,
		\]
		for some $\gamma\in L^2(\Gamma_{\cJ^c})$.
	\end{itemize}
\end{corollary}

\subsubsection{Tensor products of metaplectic operators}

 For $\hat S_1,\hat S_2\in\Mp(d,\bR)$, we consider the unique metaplectic operator $\hat S\in\Mp(2d,\bR)$ such that $\hat S(f_1\otimes f_2)=\hat S_1f_1\otimes \hat S_2f_2$ for every $f_1,f_2\in L^2(\rd)$, cf. \cite[Appendix B]{CGSymplectic}. The related symplectic matrix is 
\[
	S=\left(\begin{array}{cc|cc}
		A_1 & 0_{d} & B_1 & 0_{d}\\
		0_{d} & A_2 & 0_{d} & B_2\\
			\hline
		C_1 & 0_{d} & C_2 & 0_{d}\\
		0_{d} & C_2 & 0_{d} & D_2
	\end{array}\right),
\]
where $A_j,\ldots,D_j$ are the blocks of $S_j=\pi^{Mp}(\hat S_j)$, $j=1,2$. Consequently, $B=\diag(B_1,B_2)$, 
\[
	R(B)=R(B_1)\times R(B_2)
\]
and, consequently,
\[
	\ker(B)^\perp=R(B^T)=R(B_1^T)\times R(B_2^T)=\ker(B_1)^\perp\times\ker(B_2)^\perp,
\]
and
\[
A(\ker(B))=A_1(\ker(B_1))\times A_2(\ker(B_2)).
\]
Moreover, if $r_j=rank(B_j)$, $j=1,2$, then $r=rank(B)=r_1+r_2$. Let us consider $M_j,N_j\in\Sym_+(d,\bR)$, $j=1,2$, with
\[
	\ker(M_j)=\ker(B_j) \quad \text{and} \quad R(N_j)=R(B_j),
\]
and set $M=\diag(M_1,M_2),N=\diag(N_1,N_2)\in\Sym_+(2d,\bR)$. Then,
\[
	\ker(M)=\ker(B) \quad \text{and} \quad R(N)=R(B),
\]
the positive eigenvalues of $MB^TNB$ are $\lambda_1,\ldots,\lambda_{r_1},\mu_1,\ldots\mu_{r_2}$, where $\lambda_1,\ldots,\lambda_{r_1}$ are the positive eigenvalues of $M_1B^T_1N_1B_1$ and $\mu_1,\ldots,\mu_{r_2}$ are the positive eigenvalues of $M_2B_2^TN_2B_2$. 

In this setting Theorem \ref{Hardyglobale} can be rephrased as:
\begin{corollary}
	Let $f\in L^2(\rdd)\setminus\{0\}$. Assume that for almost every $(x_2,y_2)\in\ker(B_1)\times\ker(B_2)$ and almost every $(\xi_2,\eta_2)\in A_1(\ker(B_1))\times A_2(\ker(B_2))$ there exist $\alpha(x_2,y_2),\beta(\xi_2,\eta_2)>0$ such that:
	\[\begin{split}
		&|f(x,y)|\leq \alpha(x_2,y_2)e^{-\pi M_1x_1\cdot x_1}e^{-\pi M_2y_1\cdot y_1}, \qquad (x_1,y_1)\in\ker(B_1)^\perp\times\ker(B_2)^\perp,\\
		&|\hat Sf(\xi,\eta)|\leq \beta(\xi_2,\eta_2)e^{-\pi N_1\xi_1\cdot\xi_1}e^{-\pi N_2\eta_1\cdot\eta_1}, \qquad (\xi_1,\eta_1)\in R(B_1)\times R(B_2).
	\end{split}\]
	If $\lambda_j>1$ or $\mu_k>1$ for some $j$ and $k$, then $f=0$ almost everywhere.
\end{corollary}

\section{Dynamical versions of Hardy's
	uncertainty principle}\label{sec:DH}
	
	\subsection{Schr\"odinger evolutions with quadratic Hamiltonians}
We apply the theory above to generalize the so-called dynamical version of the Hardy's uncertainty principle exhibited by  Knutsen \cite[Theorem 3.1]{Knutsen}.

First, we exhibit how the symplectic and metaplectic group relate to the Schr\"{o}dinger equation with quadratic Hamiltonian $H$, following the terminology in \cite{Knutsen}. The general Hamiltonian equations determining the time evolution of a point $z=(x, \xi)$ is
\begin{equation}\label{eh1}
	\frac{\partial x_j}{\partial t} = \frac{\partial H}{\partial \xi_j}\quad
	\mbox{and} \quad
	\frac{\partial \xi_j}{\partial t} = -\frac{\partial H}{\partial x_j}
\end{equation}
for $j = 1, 2, \dots, d$. More compactly,
\[
\frac{\partial z}{\partial t} = J \nabla H
\]
with 
\[
\nabla = \left( \frac{\partial}{\partial x_1}, \dots, \frac{\partial}{\partial x_d}, \frac{\partial}{\partial \xi_1}, \dots, \frac{\partial}{\partial \xi_d} \right).
\]
If we consider the Hamiltonian of the form $H(z) = \frac{1}{2} \langle \cM z, z \rangle$ in \eqref{Hamilt}, where $\cM$ is a real-symmetric matrix, the Hamiltonian equations reduce to
\[
\frac{\partial z}{\partial t} = J \cM z.
\]
Starting at time $t = 0$ from the point $z(0) = z_0$, the solution to \eqref{eh1}   is
\[
z(t) = S_t(z_0) \quad \text{with} \quad S_t := \exp(t J \cM).
\]
We now consider the \emph{symplectic algebra}  $\mathfrak{sp}(d, \mathbb{R})$, that is the Lie algebra of the symplectic group $\mbox{Sp}(d,\bR)$, which consists of all matrices $X$ such that
\[
X J + J X^T = 0.
\]
The matrix $J M$ (and also $t J M$, for every $t\in\bR$) belongs to $\mathfrak{sp}(d, \mathbb{R})$. Conversely, for any $X \in \mathfrak{sp}(d, \mathbb{R})$, we have that $J X$ is symmetric. Hence, any quadratic Hamiltonian $H$ can be expressed as
\[
H(z) = -\frac{1}{2} \langle J X z, z \rangle \quad \text{for some} \quad X \in \mathfrak{sp}(d, \mathbb{R}). \tag{2.10}
\]
Since the exponential maps the symplectic algebra into the symplectic group,  the operator $S_t = \exp(t X)$ is symplectic for any fixed $t$, and, in turn, the family $\{S_t\}_{t \geq 0}$ is a one-parameter subgroup of $\text{Sp}(d, \mathbb{R})$.

Denoting by $\mathfrak{mp}(2n, \mathbb{R})$ the Lie algebra of the metaplectic group (the metaplectic algebra),  we observe  that the metaplectic algebra $\mathfrak{mp}(d, \mathbb{R})$ is isomorphic to $\mathfrak{sp}(d, \mathbb{R})$, which follows from the fact that $\text{Mp}(d, \mathbb{R})$ is a two-fold covering group of $\text{Sp}(d, \mathbb{R})$ \cite[Chapter 15]{Gos11}. In particular, there is an explicit isomorphism 
\[
F: \mathfrak{sp}(d, \mathbb{R}) \to \mathfrak{mp}(d, \mathbb{R})
\]
so that the diagram

\begin{center}
	\begin{tikzcd}
		\mathfrak{mp}(d, \mathbb{R}) \arrow[r,"F^{-1}"] \arrow[d,"\exp"] & \mathfrak{sp}(d, \mathbb{R})\arrow[d,"\exp"] \\
		\Mp(d,\bR) \arrow[r, "\pi^{Mp}"] & \Sp(d,\bR)
	\end{tikzcd}
	\end{center}
 commutes.
Therefore, we can describe the solution of the Schr\"{o}dinger equation \eqref{eqScrod} with quadratic Hamiltonians as a lift of the flow $t \mapsto S_H^t$ in $\text{Sp}(d, \mathbb{R})$ into the unique path $t \mapsto \widehat{S}_t$ in $\text{Mp}(d, \mathbb{R})$ so that $\widehat{S}_0 = I_d$.

\begin{proposition} [Corollary 2.1 in \cite{Knutsen}] Let the Hamiltonian $H$ be quadratic, and let $t \mapsto \widehat{S}_t$ be the lift to $\text{Mp}(d, \mathbb{R})$ of the flow $t \mapsto S_t$. Then for any $u_0 \in S(\mathbb{R}^n)$, the function
	\begin{equation}\label{eh2}
		u(x, t) = \widehat{S}_t u_0(x)
	\end{equation}
	is a solution of the initial value problem
	\[
	i \hbar\frac{\partial u}{\partial t} (x, t) = {H}_D u(x, t), \quad u(x, 0) = u_0(x).
	\]
\end{proposition} 
We have now the background to extend Theorem $3.1$ in \cite{Knutsen}. 
For any $t_1\geq 0$, we assume that $\pi^{Mp}( \widehat{S}_{t_1})={S}_{t_1}$, where 
\begin{equation}\label{blockST}
	{S}_{t_1}=\begin{pmatrix}
		A_{t_1} & B_{t_1}\\
		C_{t_1} & D_{t_1}
	\end{pmatrix}.
\end{equation}
\begin{theorem}\label{main}
	Let \( u(x, t)  \) be the solution to the Schrödinger equation \eqref{eqScrod} with the quadratic Hamiltonian \eqref{Hamilt} given by $H(z) = -\frac{1}{2} \langle J X z, z \rangle$, for some  $X \in \mathfrak{sp}(d, \mathbb{R})$,  $\hbar=1/(2\pi)$, and initial datum $u_0\in \mathcal{S}(\mathbb{R}^d)$. Suppose at time $t=0$ and  $t={t_1}>0$ the solution $u$ satisfies  the decay conditions  
	\begin{align}
		\label{E1-2}
		& |u(x_1+x_2,0)|\leq \alpha(x_2)e^{-\pi Mx_1\cdot x_1}, \qquad x_1\in \ker(B_{t_1})^\perp, \ x_2\in D_{t_1}^TA_{t_1}(\ker(B_{t_1})),\\
		\label{E2-2}
		& |u(\xi_1+\xi_2,t_1)|\leq \beta(\xi_2)e^{-\pi N \xi_1\cdot \xi_1}, \qquad \xi_1\in R(B_{t_1}), \ \xi_2\in A_{t_1}(\ker(B_{t_1})),
	\end{align}
	where $B_T$ is the upper-right-side block in \eqref{blockST}, $M,N\in\Sym_+(d,\bR)$ are positive-semidefinite matrices
	satisfying \eqref{G1} and \eqref{G2} with $B=B_{t_1}$.  Assume that   $r=rank(B_{t_1})\geq 1$ ($B_{t_1}\neq0_{d}$).
	If $\lambda_1,\ldots,\lambda_r$ are the non-zero eigenvalues of  $MB^T_{t_1}NB_{t_1}$, then 
	\begin{itemize}
	\item[$(i)$] $\lambda_j\leq 1$ for every $j=1,\ldots,r$.
	\item[$(ii)$] If $\lambda_1=\ldots=\lambda_r=1$, then, 
	\begin{align}\label{Th2}
	u_0(x_1+x_2)=\gamma(x)e^{-\pi(M+iB_{t_1}^+A_{t_1})x_1\cdot x_1}e^{2\pi iC_{t_1}^TA_{t_1}x\cdot x_1}, \qquad x_1\in\ker(B_{t_1})^\perp,
	\end{align} 
	for some $\gamma\in L^2(\ker(B_{t_1}))$, where $x_2=D_{t_1}^TAx$ and $x\in\ker(B_{t_1})$.
\end{itemize}
\end{theorem}
\begin{proof}
	Observe that $u(x_1+x_2,0)=u_0(x_1+x_2)$ and, using the equality \eqref{eh2}, $u(\xi_1+\xi_2,{t_1})=\widehat S_{t_1} u_0(\xi_1+\xi_2)$. Consequently, the thesis follows from Theorem \ref{Hardyglobale}.
\end{proof}

The case $r=d$ yields Theorem \ref{1.7} in the introduction. 

\begin{remark}
	Hardy's uncertainty principle for metaplectic operators with free projections have made the proof of Theorem \ref{main}  very simple. This result improves Theorem 3.1 in \cite{Knutsen}, where the assumptions \eqref{E7} are
	simplified to
	\begin{equation}\label{W2bis}
		|u_0(x)| \lesssim   e^{-\a |x|^2}, \quad \text{and} \quad |u_1(x)|  \lesssim  e^{-\beta |x|^2}, \quad x\in \rd,
	\end{equation}
	and
	\begin{equation}\label{W4}
		(2\hbar)^2 \|\mathcal{B}({t_1})\|_{op}^2\a\beta>1,
	\end{equation}
	where $\mathcal{B}({t_1})=B_{t_1}$ in our setting and $\hbar=1/(2\pi)$.
	Using our notations:
	$$M=\frac{\alpha}{\pi}I_{d} \quad \text{and} \quad N=\frac{\alpha}{\pi}I_{d}.$$
	Recall that $\|B_{t_1}\|_{op}=\sqrt{\rho(B^T_{t_1}B_{t_1})}$ where, if $\gamma_j$, $j=1\dots d$, are the (positive) eigenvalues of $B^T_{t_1}B_{t_1}$,  the spectral radius $\rho$ is defined by
	$$\rho(B^T_{t_1}B_{t_1})=\max_{1\leq j\leq d}{|\gamma_j|}=\max_{1\leq j\leq d}\gamma_j.$$
	The assumption \eqref{W4} becomes
	\begin{equation}\label{K1}
		\rho(B^T_{t_1}B_{t_1})\alpha \beta>\pi^2.
	\end{equation}
	Since the eigenvalues $\lambda_j$ of $MB^T_{t_1}NB_{t_1}$ are given by
	$$\lambda_j=\frac{\a\beta}{\pi^2}\gamma_j,$$ 
	to get $u_0\equiv0$  Knutsen's condition \eqref{K1}  requires
	$$(\alpha \beta)\max_{1\leq j\leq d} \gamma_j>\pi^2,$$
	i.e.,
	$$\max_{1\leq j\leq d} \lambda_j>1.$$
	Our setting proves the sharpness of the above condition, c.f. $(i)$ in Theorem \ref{Hardyglobale}. Furthermore, it displays  the form of the solution $u_0$ depending on the values of $\lambda_1,\dots,\lambda_d$, cf. the theses  $(i)$ and $(ii)$ of Theorem \ref{Hardyglobale}.
\end{remark}
\subsection{Examples of Schr\"{o}dinger evolutions}
\subsubsection{Anisotropic harmonic oscillator}
A basic example which is covered by Theorem \ref{main} is the  Cauchy problem for the anisotropic  harmonic oscillator. This case, to our knowledge,  is not considered by   the previous   literature. For simplicity, we state the case  in dimension $d=2$, but it could be generalized to any dimension $d\geq2$. For  $x=(x_1,x_2)\in \R\times\R$, $t\in\R$, we study
\begin{equation}\label{5.31eq}
	\begin{cases} i  \partial_t
		u =H_D u,\\
		u(0,x)=u_0(x),
	\end{cases}
\end{equation}
where
\begin{equation}\label{5.31eqH}H_D= -\frac{1}{4\pi}\partial^2_{x_2} +\pi x_2^2.\end{equation}
By inspection the related quadratic Hamiltonian is given by
\begin{equation}\label{esanitH}
	H(x_1,x_2,\xi_1,\xi_2)=\frac{x_2^2+\xi_2^2}{2}, 
\end{equation}
and the symmetric matrix $\cM$ is 
\begin{equation}\label{Mes}
	\cM = \begin{pmatrix}0 & 0 & 0 & 0\\0&1&0&0\\
		0&0&0&0\\
		0&0&0&1\end{pmatrix}.
\end{equation}
The corresponding element $X$ in the symplectic algebra is
\begin{equation}\label{Mes2}
	X =J\cM= \begin{pmatrix}0 & 0 & 0 & 0\\0&0&0&1\\
		0&0&0&0\\
		0&-1&0&0\end{pmatrix}\in \mathfrak{sp}(1, \mathbb{R}).
\end{equation}
For shortness, we define
\[ \Omega= \begin{pmatrix}0 & 0 \\0&1\end{pmatrix}.\]
For the power series $\exp(tX) = \sum_{k=0}^{\infty} \frac{t^k}{k!} X^k$, we distinguish between the matrices with even and odd exponents. By induction and observing that $	\Omega^{k}=	\Omega$, for every $k\in\bN_+$,  they are given by
\[
X^{2k} = (-1)^k
\begin{pmatrix}
	\Omega & 0 \\
	0 & \Omega
\end{pmatrix}\quad \mbox{and}\quad X^{2k+1} = (-1)^k
\begin{pmatrix}
	0 & \Omega  \\
	-\Omega & 0
\end{pmatrix}.
\]
Thus,
\[
\exp(tX) = \sum_{k=0}^{\infty} \frac{(-1)^k t^{2k}}{(2k)!}
\begin{pmatrix}
	\Omega & 0 \\
	0 & \Omega
\end{pmatrix}
+ \sum_{k=0}^{\infty} \frac{(-1)^k t^{2k+1}}{(2k+1)!}
\begin{pmatrix}
	0 & \Omega\\
	-\Omega& 0
\end{pmatrix}.
\]
Moving the summation inside the matrix (since $\Omega$ is diagonal), and writing  $$\exp(tX) = S_H^t=
\begin{pmatrix}
	A_t & B_t \\
	C_t & D_t
\end{pmatrix}$$  we obtain
\begin{equation}\label{B(t)}
	B_t =  \sum_{k=0}^{\infty} \frac{(-1)^k t^{2k+1}}{(2k+1)!} \Omega = \begin{pmatrix}
		0 & 0 \\
		0 & \sum_{k=0}^{\infty} \frac{(-1)^k t^{2k+1}}{(2k+1)!} 
	\end{pmatrix}
	= \begin{pmatrix}
		0 & 0 \\
		0 & \sin t 
	\end{pmatrix}.
\end{equation}
Using the Taylor series for the sine and cosine functions we work out the other blocks:
\[
A_t = D_t =\begin{pmatrix}
	0 & 0 \\
	0 & \cos t 
\end{pmatrix}\quad \text{and} \quad C_t = -B_t.
\]
Thus, the solution  $u(t,x)=e^{-itH}u_0$, with initial datum $u_0$  in $\cS(\R^2)$, 
has the propagator  $e^{-itH}=\widehat{S}_H^t $, which is a one-parameter family of  metaplectic operators with  symplectic matrices
\begin{equation}\label{esempio}
	S_H^t = \left(\begin{array}{cc|cc}1 & 0 & 0 & 0\\0&\cos t&0&\sin t\\
		\hline
		0&0&1&0\\
		0&-\sin t&0&\cos t\end{array}\right)\quad t\in\R.
\end{equation}
Observe that the $B_t$-block \eqref{B(t)} satisfies $B_t\not=0$ if and only if $t\not= k\pi$, $k\in\bN$ (we consider $t\geq 0$). Moreover, for $t\not=k\pi$, 
\[
\ker(B_t)=\mbox{span}\left\{ \begin{pmatrix}1 \\ 0\end{pmatrix} \right\} \quad \Rightarrow \quad A_t\left(\ker(B_t\right))=\ker(B_t),
\]
whereas:
\[
R(B_t)=\mbox{span}\left\{\begin{pmatrix}0\\ 1\end{pmatrix}\right\} \quad \Rightarrow \quad R(B_t)=\ker(B_t)^\perp.
\]
The matrices $M,N\in \Sym_+(1,\bR)$ can be written as
$$ M=\begin{pmatrix}
	0 & 0 \\
	0 & a
\end{pmatrix}\quad\mbox{and}\quad N=\begin{pmatrix}
	b & 0 \\
	0 & 0
\end{pmatrix}$$
for any $a,b>0$.  For $t\not=k \pi$, the pseudo-inverse $B^+_t$ is 
\[
B^+_t=\begin{pmatrix}
	0 & 0 \\
	0 & \frac1{\sin t}
\end{pmatrix}\quad\mbox{and}\quad B_t^+A_t=
\begin{pmatrix}
	0 & 0 \\
	0 & \frac{\cos t}{\sin t}
\end{pmatrix}.\]
Furthermore,  the eigenvalues of $MB_t^TNB_t$ are given by $\lambda_1=0$ and $\lambda_2=ab\sin^2t$. 

Theorem \eqref{main} applies to this case and provides  the following Hardy-type estimate:
\begin{corollary}\label{e00}
	Let $u(\cdot, t) \in \cS(\mathbb{R}^2)$ be the solution of the Schr\"{o}dinger equation \eqref{5.31eq}. Suppose at time $t = 0$ and time $t = {t_1}>0$, the solution $u$ satisfies the decay conditions
	\[
	|u(x,y, 0)| \leq \alpha(x) e^{-a\pi y^2} \quad \text{and} \quad |u(x,y, {t_1})| \leq \beta(x) e^{-b\pi y^2},
	\]
	for every $x,y\in\bR$, for some constants $a,b>0$ and where $\alpha,\beta:\bR\to (0,+\infty)$ are positive functions. Then,
	\begin{itemize}
	\item[$(i)$] If $ab\sin^2{t_1}>1$ then $u\equiv0$.
	\item[$(ii)$] If $ab\sin^2{t_1}=1$, then 
	$$u_0(x,y)=\gamma(x)e^{-\pi (a+i\frac{\cos {t_1}}{\sin {t_1}}+i\sin(2{t_1}))y^2},\quad x,y\in\bR,$$
	for some function $\gamma\in L^2(\bR)$.
	\end{itemize}
\end{corollary}

Observe that the condition $ab\sin^2{t_1}\geq1$  implies ${t_1}\neq k\pi$.

\subsubsection{Harmonic Oscillator}
The standard harmonic oscillator was treated by Knutsen \cite[Example 4.2]{Knutsen}. The Hamiltonian is
\[
H(z) = \frac{1}{2m} \left( \xi_1^2 + \cdots + \xi_d^2 \right) + \frac{m}{2} \left( \omega_1^2 x_1^2 + \cdots + \omega_d^2 x_d^2 \right),
\]
where \( z = (x, \xi) \). The associated Schrödinger equation reads
\begin{equation}
i \hbar \frac{\partial u}{\partial t}(x, t) = \left( - \frac{\hbar^2}{2m} \Delta + \frac{m}{2} \left( \omega_1^2 x_1^2 + \cdots + \omega_d^2 x_d^2 \right) \right) u(x, t). \label{4.2}
\end{equation}
The matrix $\exp(tX)$ was computed in \cite[Example 4.2]{Knutsen}. The 
 \( B_t \)-block reads
\[
B_t=  \frac{1}{m} \operatorname{diag} \left( \frac{\sin(\omega_j t)}{\omega_j} \right).
\]
The other blocks are
\[
A_t= D_t = \operatorname{diag} \left( \cos(\omega_j t) \right)
\quad \text{and} \quad
C_t= -m \operatorname{diag} \left( \omega_j \sin(\omega_j t) \right).
\]
In \cite[Example 4.2]{Knutsen} only the free case was considered, that is the values ${t_1}$ for which $ \frac{\sin(\omega_j {t_1})}{\omega_j} \not=0$ for all $j\in\{1,\dots,d\}$.  Here we can improve Corollary 4.2. in \cite{Knutsen} as follows.
\begin{corollary}
Let \( u(\cdot, t) \in \mathcal{S}(\mathbb{R}^d) \) be the solution of the Schrödinger equation \eqref{4.2} corresponding to the harmonic oscillator $\hbar=1$, $m=1/2$. Suppose at time \( t = 0 \) and time \( t = {t_1} \), the solution \( u \) satisfies the decay conditions
\[
|u(x, 0)| \leq K e^{-\alpha |x|^2} \quad \text{and} \quad |u(x, {t_1})| \leq K e^{-\beta |x|^2}
\]
for some constants \( \alpha, \beta, K > 0 \). If 
\[
16\alpha \beta\max_j \left| \frac{\sin(\omega_j {t_1})}{\omega_j} \right|^2 > 1,
\]
then \( u \equiv 0 \).
\end{corollary}

\section*{Acknowledgements}
Elena Cordero and Gianluca Giacchi are members of the GNAMPA group, INDAM. Eugenia Malinnikova is partly supported by NSF grant, DMS-2247185. 

\begin{appendix}
\renewcommand{\thetable}{\Alph{section}\arabic{table}}
\section{}
Here, we state a linear algebra tool, which is already known in the literature, and report its proof for the sake of completeness.
\begin{lemma}\label{autovalori positivi}
	(i) If $M,N\in \Sym_{+}(d,\bR)$ 
 and $B\in\bR^{d\times d}$, then $MB^TNB$ has non-negative eigenvalues.\\
	(ii) If  $M,N\in\Sym_{++}(d,\bR)$ 
 and $B\in\GL(d,\bR)$, then the matrix $MB^TNB$ has positive eigenvalues.
\end{lemma}
\begin{proof}
	If $M$ and $N$ are positive-semidefinite, then $B^TNB$ is positive-semidefinite for every $B\in\bR^{d\times d}$. By \cite[Theorem 7.6.2]{HJ}, $MB^TNB$ has non-negative eigenvalues. This proves $(i)$. 
	
	To prove $(ii)$, observe that if $M_1$ and $M_2$ are positive-definite $d\times d$ matrices, then the eigenvalues of $M_1M_2$ are real and positive. In fact, as already observed in \cite[Sec. 2.1]{CGP2024}, the product matrix $M_1M_2$ has the same eigenvalues as the  positive-definite matrices $M_1^{1/2}M_2M_1^{1/2}$ and $M_2^{1/2}M_1M_2^{1/2}$. Consequently, if $M$ and $N$ are positive-definite and $B\in\GL(d,\bR)$, then $B^TNB$ is also positive-definite and the product $MB^TNB$ has positive eigenvalues. 
	
\end{proof}

\section{}
	In this section of the Appendix, we list many important relations between the blocks of a symplectic matrix. The result from which they can be inferred is the following restatement of \cite[Property 1]{MO2002}, using that $R(B^T)=\ker(B)^\perp$.
	
	\begin{lemma}\label{lemmaB1}
		Let $S\in\Sp(d,\bR)$ have block decomposition \eqref{blockS}. Then,
		\begin{itemize}
		\item[$(i)$] $D^{-T}(\ker(B)^\perp)=R(B)$, where $D^{-T}$ denotes the pre-image under $D^T$,
		\item[$(ii)$] $D(\ker(B))=R(B)^\perp$.
		\end{itemize}
	\end{lemma}
	Lemma \ref{lemmaB1} has three straightforward implications.
	\begin{corollary}\label{corB2}
		Let $S\in\Sp(d,\bR)$ have block decomposition \eqref{blockS}. Then,
		\begin{itemize}
		\item[$(i)$] $D^T(R(B))\subseteq \ker(B)^\perp$.
		\item[$(ii)$] $D:\ker(B)\to R(B)^\perp$ is an isomorphism.
		\item[$(iii)$] $R(B)^\perp\subseteq R(D)$.
		\end{itemize}
	\end{corollary}
	\begin{proof}
		$(i)$ Since $D^{-T}(\ker(B)^\perp)=R(B)$ and $D^T(D^{-T}(\ker(B)^\perp))\subseteq \ker(B)^\perp$, we find $D^T(R(B))\subseteq\ker(B)^\perp$.\\
		$(ii)$ $\ker(B)$ and $R(B)^\perp$ have the same dimension and $D:\ker(B)\to R(B)^\perp$ is surjective by Lemma \ref{lemmaB1} $(ii)$. Consequently, $D$ is also injective.\\
		$(iii)$ Let $\xi\in R(B)^\perp$. By the surjectivity claimed in $(ii)$, there exists $x\in\ker(B)$ such that $\xi=Dx$. Hence, $\xi\in R(D)$.
		
	\end{proof}
	
	Analogous considerations can be obtained for other blocks of $S$, by applying Corollary \ref{corB2} to $SJ$, $JS$, $JSJ$ and their inverses, as illustrated in Table \ref{table:1}. 

\begin{table}
\centering
\begin{tabular}{ccc}
\hline
\hline
 & Reference matrix & Corresponding relations\\
\hline
\hline
\multirow{3}{*}{$(i)$} & \multirow{3}{*}{$S=\begin{pmatrix} A & B\\ C & D\end{pmatrix}$} & $D^T(R(B))\subseteq \ker(B)^\perp$.\\
& & $D:\ker(B)\to R(B)^\perp$ is an isomorphism.\\
& & $R(B)^\perp\subseteq R(D)$.\\
\hline
\multirow{3}{*}{$(ii)$} & \multirow{3}{*}{$SJ=\begin{pmatrix} -B & A\\ -D & C\end{pmatrix}$} & $C^T(R(A))\subseteq\ker(A)^\perp$.\\
& & $C:\ker(A)\to R(A)^\perp$ is an isomorphism.\\
& & $R(A)^\perp\subseteq R(C)$.\\
\hline
\multirow{3}{*}{$(iii)$} & \multirow{3}{*}{$JS=\begin{pmatrix} C & D\\ -A & -B\end{pmatrix}$} & $B^T(R(D))\subseteq\ker(D)^\perp$.\\
& & $B:\ker(D)\to R(D)^\perp$ is an isomorphism.\\
& & $R(D)^\perp\subseteq R(B)$.\\
\hline
\multirow{3}{*}{$(iv)$} & \multirow{3}{*}{$JSJ=\begin{pmatrix} -D & C\\ B & -A\end{pmatrix}$} & $A^T(R(C))\subseteq \ker(C)^\perp$.\\
& & $A:\ker(C)\to R(C)^\perp$ is an isomorphism.\\
& & $R(C)^\perp\subseteq R(A)$.\\
\hline
\multirow{3}{*}{$(v)$} & \multirow{3}{*}{$S^{-1}=\begin{pmatrix} D^T & -B^T\\ -C^T & A^T\end{pmatrix}$} & $A(\ker(B)^\perp)\subseteq R(B)$.\\
& & $A^T:R(B)^\perp\to \ker(B)$ is an isomorphism.\\
& & $\ker(B)\subseteq \ker(A)^\perp$.\\
\hline
\multirow{3}{*}{$(vi)$} & \multirow{3}{*}{$S^{-1}J=\begin{pmatrix} B^T & D^T \\ -A^T & -C^T\end{pmatrix}$} & $C(\ker(D)^\perp)\subseteq R(D)$.\\
& & $C^T:R(D)^\perp\to \ker(D)$ is an isomorphism.\\
& & $\ker(D)\subseteq \ker(C)^\perp$.\\
\hline
\multirow{3}{*}{$(vii)$} & \multirow{3}{*}{$JS^{-1}=\begin{pmatrix} -C^T & A^T\\ -D^T & B^T \end{pmatrix}$} & $B(\ker(A)^\perp)\subseteq R(A)$.\\
& & $B^T:R(A)^\perp\to \ker(A)$ is an isomorphism.\\
& & $\ker(A)\subseteq \ker(B)^\perp$.\\
\hline
\multirow{3}{*}{$(viii)$} & \multirow{3}{*}{$JS^{-1}J=\begin{pmatrix} -A^T & -C^T \\ -B^T & -D^T \end{pmatrix}$} & $D(\ker(C)^\perp)\subseteq R(C)$.\\
& & $D^T:R(C)^\perp\to \ker(C)$ is an isomorphism.\\
& & $\ker(C)\subseteq \ker(D)^\perp$.\\
\hline
\hline
\end{tabular}

\caption{Block relations derived by applying Corollary \ref{corB2} to $S$ and the reference matrices listed in the second column.}
\label{table:1}
\end{table}

\section{}

	Here we prove the  following:
	\begin{lemma}\label{lemma52}
	Let $S\in\Sym(d,\bR)$ be such that $B\neq0_d$ in its block decomposition \eqref{blockS}. Let $M,N\in\Sym_{+}(d,\bR)$ be such that \eqref{G1} and \eqref{G2} hold, i.e.,
	\begin{equation}\label{2728bis}
		\ker(M)=\ker(B), \qquad R(N)=R(B).
	\end{equation}
	Let $r=rank(B)$, and $V:\bR^r\to \ker(B)^\perp$ be a parametrization of $\ker(B)^\perp$. Assume that $V^TV=I_r$.
	The following statements are equivalent:
	\begin{itemize}
	\item[$(i)$] $\lambda$ is an eigenvalue of $V^TMB^TNBV$.
	\item[$(ii)$] $\lambda$ is an eigenvalue of $MB^TNB|_{\ker(B)^\perp}$.
	\item[$(iii)$] $\lambda$ is a non-zero eigenvalue of $MB^TNB$.
	\end{itemize}
\end{lemma}

 Let us collect a few remarks.

	\begin{remark}\label{remarkPreProof}
(a) We are considering the action of $MB^TNB\in\bR^{d\times d}$ only on $\ker(B)^\perp$.\\
(b) The assumptions \eqref{2728bis} of Lemma \ref{lemma52} are not merely conditions under which the conclusion of Lemma \ref{lemma52} follows easily. They are indeed necessary for $MB^TNB$ to map $\ker(B)^\perp$ to itself. Actually, under these assumptions, $MB^TNB$ is an isomorphism of $\ker(B)^\perp$ to itself, as composition of isomorphisms. Namely,
	\[
		\begin{tikzcd}[row sep=huge, column sep=large]	
			\ker(B)^\perp \arrow[r,"B"] & R(B)=R(N)=R(N^T)=\ker(N)^\perp,
		\end{tikzcd}
	\]
	so that:
	\[
		\begin{tikzcd}[row sep=huge]
			\ker(N)^\perp \arrow[r,"N"] & R(N)=R(B)=\ker(B^T)^\perp \arrow[r,"B^T"] & R(B^T)=\ker(B)^\perp=\ker(M)^\perp,
		\end{tikzcd}
	\]
	and, finally, 
	\[
		\begin{tikzcd}[row sep=huge, column sep=large]	
			\ker(M)^\perp \arrow[r,"M"] & R(M)=R(M^T)=\ker(M)^\perp=\ker(B)^\perp.
		\end{tikzcd}
	\]
	(d) The eigenvalues $\{\lambda_1,\ldots,\lambda_r\}$ of $MB^TNB|_{\ker(B)^\perp}$ are all positive real numbers, since they are non-negative and $MB^TNB|_{\ker(B)^\perp}$ is an isomorphism.
\end{remark}

\begin{proof}[\textbf{Proof of Lemma \ref{lemma52}}]
			We prove that $(i)\Rightarrow(ii)$. Let $\lambda$ be an eigenvalue of $V^TMB^TNV$ and $x\in\bR^r\setminus\{0\}$ be such that $V^TMB^TNBVx=\lambda x$. We show that there exists $y\in \ker(B)^\perp$ so that $MB^TNBy=\lambda y$. For, let $y=Vx$. Since $R(V)=\ker(B)^\perp$, $y\in \ker(B)^\perp\setminus\{0\}$. We have:
	\begin{align*}
		MB^TNBy=MB^TNBVx.
	\end{align*}
	Since $R(M)=\ker(B)^\perp$, as displayed in Remark \ref{remarkPreProof} $(b)$, and $VV^T$ is the identity on $\ker(B)^\perp$,
	\[
		MB^TNBVx=VV^TMBNB^TVx,
	\]
	so that:
	\begin{align*}
		MB^TNBy&=MB^TNBVx\\
		&=V(V^TMB^TNBVx)\\
		&=V(\lambda x)\\
		&=\lambda Vx\\
		&=\lambda y.
	\end{align*}
	Next, we prove that $(ii)\Rightarrow(i)$. For, let $\lambda$ be an eigenvalue of $MB^TNB|_{\ker(B)^\perp}$, and let $y\in \ker(B)^\perp\setminus\{0\}$ be such that $MB^TNBy=\lambda y$. Consider $x=V^Ty\in\bR^r\setminus\{0\}$. Since $V$ is a parametrization of $\ker(B)^\perp$, 
	\[
		\ker(V^T)=R(V)^\perp=\ker(B),
	\]
	so $VV^T$ is the identity on $\ker(B)^\perp$. Therefore,
	\begin{align*}
		V^TMB^TNBVx&=V^TMB^TNBVV^Ty\\
		&=V^TMB^TNBy\\
		&=V^T(\lambda y)\\
		&=\lambda (V^T y)\\
		&=\lambda x.
	\end{align*}
	This concludes the proof.
	\end{proof}
\end{appendix}

\bibliographystyle{abbrv}

\end{document}